\documentclass[reqno]{amsart}
\usepackage{amssymb}
\usepackage{graphicx}

\usepackage[usenames, dvipsnames]{color}
\usepackage{verbatim}
\usepackage{mathrsfs}
\usepackage{bm}

\numberwithin{equation}{section}

\newtheorem{theorem}{Theorem}[section]
\newtheorem{corollary}[theorem]{Corollary}
\newtheorem{lemma}[theorem]{Lemma}
\newtheorem{prop}[theorem]{Proposition}

\theoremstyle{definition}
\newtheorem{remark}[theorem]{Remark}

\theoremstyle{definition}

\theoremstyle{definition}

\makeatletter
\def\dashint{\operatorname%
{\,\,\text{\bf-}\kern-.98em\DOTSI\intop\ilimits@\!\!}}
\makeatother

\def\\det{\text{det}}

\def\.5{\frac{1}{2}}

\newcommand{\RN}[1]{%
  \textup{\uppercase\expandafter{\romannumeral#1}}%
}

\renewcommand{\epsilon}{\varepsilon}

\newcounter{marnote}

\begin{document}
\title[Optimal boundary gradient estimates ]{Optimal boundary gradient estimates for Lam\'{e} systems with partially infinite coefficients }
\author[J.G. Bao]{Jiguang Bao}
\address[J.G. Bao]{School of Mathematical Sciences, Beijing Normal University, Laboratory of MathematiCs and Complex Systems, Ministry of Education, Beijing 100875, China.}
\email{jgbao@bnu.edu.cn}
\thanks{J.G. Bao was partially supported by NSFC (11371060), and Beijing Municipal Commission of Education for the Supervisor of Excellent Doctoral Dissertation (20131002701). }
\author[H. J. Ju]{Hongjie Ju}
\address[H.J. Ju]{School of  Sciences, Beijing University of Posts  and Telecommunications,
Beijing 100876, China}
\email{hjju@bupt.edu.cn}
\thanks{H.J. Ju was partially supported by NSFC (11301034) (11471050).}

\author[H.G. Li]{Haigang Li}
\address[H.G. Li]{School of Mathematical Sciences, Beijing Normal University, Laboratory of MathematiCs and Complex Systems, Ministry of Education, Beijing 100875, China.}
\email{hgli@bnu.edu.cn}
\thanks{H.G. Li was partially supported by  NSFC (11571042), Fok Ying Tung Education Foundation (151003), and the Fundamental Research Funds for the Central Universities.}


\date{\today} 


\maketitle

\begin{abstract}
In this paper, we derive the pointwise upper bounds and lower bounds on the gradients of solutions to the Lam\'{e} systems with partially infinite coefficients as the surface of discontinuity of the coefficients of the system is located very close to the boundary. When the distance tends to zero, the optimal blow-up rates of the gradients are established for inclusions with arbitrary shapes and in all dimensions.
\end{abstract}

\section{Introduction and main results}

It is a common phenomenon that high concentration of extreme mechanical loads occurs in high-contrast fiber-reinforced composites in the zones that  include the narrow regions between two adjacent inclusions and the thin gaps between the inclusions and the exterior boundary of the background medium. Extreme loads are always amplified by such composite microstructure, which will cause failure or fracture initiation. Stimulated by the well-known work of Babu\u{s}ka et al \cite{BASL}, where computational analysis of damage and fracture in fiber composite systems is investigated, we consider the Lam\'{e} system in linear elasticity with partially infinite coefficients to characterize the high-contrast composites. This paper is a continuation of \cite{BLL,BLL2}, where the upper bound of the gradient estimate for two adjacent inclusions is established, which can be regarded as interior estimates for this problem. 

Due to the interaction from the boundary data, solutions of these systems become more irregular near the boundary. In this paper, we mainly investigate the boundary gradient estimates for the Lam\'{e} system with partially infinite coefficients  when the inclusion is spaced very close to the matrix exterior boundary. The novelty of these estimates is that they give not only the pointwise upper bounds but also lower bounds of the gradient, which shows that the blow-up rate of the gradient with respect to the distance between the inclusion and the matrix exterior boundary that we obtain is optimal. The role of the boundary data is embodied in these estimates. Especially, an explicit factor that determines whether the blow-up occurs or not is singled out in the lower bound estimates. We would like to emphasize that the gradient estimates obtained in this paper hold for inclusions with arbitrary convex shapes and in all dimensions. 

Let $D\subset\mathbb{R}^{d}(d\geq2)$ be a bounded open set with $C^{2, \gamma}$ boundary, and $D_{1}$ be a strictly convex open set in $D$
with $C^{2, \gamma}$ boundary, $0<\gamma<1$, and spaced very close to the boundary $\partial{D}$. More precisely,
\begin{equation}\label{Domain}
\begin{array}{l}
\displaystyle \overline{D}_{1}\subset{D},\quad \hbox{the principle curvatures of }\ \partial{D},\partial{D}_{1}\geq \kappa_0>0,\\
\displaystyle \varepsilon:=dist (D_{1}, \partial{D})>0,
\end{array}
\end{equation}
where $\kappa_0$ is constant independent of $\epsilon$. We also assume that the $C^{2,\gamma}$ norms of $\partial{D}_{1}$ are bounded by some constant independent of $\varepsilon$. This implies that $D_{1}$ contains a ball of radius $r_{0}^{*}$ for some constant $r_{0}^{*}>0$ independent of $\varepsilon$. See Figure \ref{fig:1.1}.

\begin{figure}[t]
\centering
\includegraphics[width=2.0in]{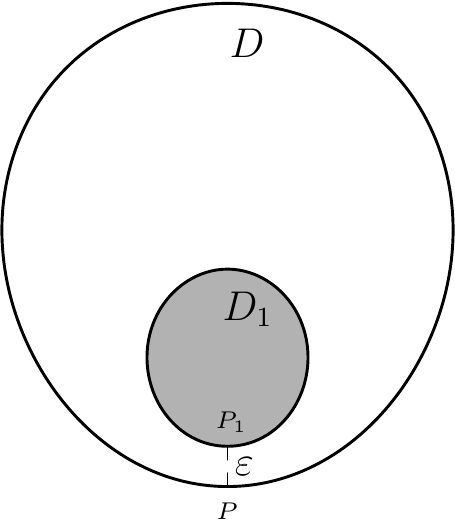}
\caption{\small One inclusion close to the boundary.}
\label{fig:1.1}
\end{figure}

Denote $$\Omega:=D\setminus\overline{D}_{1}.$$
We assume that $\Omega$ and $D_{1}$ are occupied, respectively, by two different isotropic and homogeneous materials with different Lam\'{e} constants $(\lambda, \mu)$ and $(\lambda_1, \mu_1)$. Then the elasticity tensors for the background and the inclusion can be written, respectively, as $\mathbb{C}^0$ and $\mathbb{C}^1$, with
$$C_{ijkl}^0=\lambda\delta_{ij}\delta_{kl} +\mu(\delta_{ik}\delta_{jl}+\delta_{il}\delta_{jk}),$$
and
$$C_{ijkl}^1=\lambda_1\delta_{ij}\delta_{kl} +\mu_1(\delta_{ik}\delta_{jl}+\delta_{il}\delta_{jk}),$$
where $i, j, k, l=1,2,\cdots,d$ and $\delta_{ij}$ is the kronecker symbol: $\delta_{ij}=0$ for $i\neq j$, $\delta_{ij}=1$ for $i=j$.

Let $u=(u^1, u^2,\cdots,u^{d})^T:D\rightarrow\mathbb{R}^{d}$ denote the displacement field. For a given vector valued function $\varphi=(\varphi^1,\varphi^2,\cdots,\varphi^{d})^{T}$, we consider the following Dirichlet problem for the Lam\'{e} system:
 \begin{align}\label{Lame}
\begin{cases}
\nabla\cdot \left((\chi_{\Omega}\mathbb{C}^0+\chi_{D_{1}}\mathbb{C}^1)e(u)\right)=0,&\hbox{in}~~D,\\
u=\varphi, &\hbox{on}~~\partial{D},
\end{cases}
\end{align}
where $\chi_{\Omega}$ is the characteristic function of $\Omega\subset \mathbb{R}^{d}$,
$$e(u)=\frac{1}{2}(\nabla u+(\nabla u)^T)$$
is the strain tensor.

Assume that the standard ellipticity condition holds for (\ref{Lame}), that is,
\begin{align*}
\mu>0,\quad d\lambda+2\mu>0,\quad \mu_1>0,\quad d\lambda_1+2\mu_1>0.
\end{align*}
For $\varphi\in H^1(D; \mathbb{R}^{d})$, it is well knowm that there exists a unique solution $u\in H^1(D; \mathbb{R}^{d})$ to the Dirichlet problem (\ref{Lame}), which is also the minimizer of the energy functional
$$J_1[u]:=\frac{1}{2}\int_\Omega \left((\chi_{\Omega}\mathbb{C}^0+\chi_{D_{1}}\mathbb{C}^1)e(u), e(u)\right)dx $$
on 
\begin{align*}
H^1_\varphi(D; \mathbb{R}^{d}):=\left\{u\in  H^1(D; \mathbb{R}^{d})~\big|~ u-\varphi\in  H^1_0(D; \mathbb{R}^{d})\right\}.
\end{align*}

We introduce the linear space of rigid displacement in $\mathbb{R}^{d}$:
$$\Psi:=\{\psi\in C^1(\mathbb{R}^{d}; \mathbb{R}^{d})\ |\ \nabla\psi+(\nabla\psi)^T=0\}.$$
With $e_{1},\cdots,e_{d}$ denoting the standard basis of $\mathbb{R}^{d}$, $$\left\{~e_{i},~x_{j}e_{k}-x_{k}e_{j}~\big|~1\leq\,i\leq\,d,~1\leq\,j<k\leq\,d~\right\}$$ is a basis of $\Psi$. Denote this basis of $\Psi$ as $\left\{\psi_{\alpha}~\big|~\alpha=1,2,\cdots,\frac{d(d+1)}{2}\right\}$.

For fixed $\lambda$ and $\mu$ satisfying $\mu>0$ and $d\lambda+2\mu>0$, denote $u_{\lambda_1,\mu_1}$ as the solution of (\ref{Lame}). Then similarly as in the Appendix of \cite{BLL}, we also have
\begin{align*}
u_{\lambda_1,\mu_1}\rightarrow u\quad\hbox{in}\ H^1(D; \mathbb{R}^{d}),\quad \hbox{as}\ \min\{\mu_1, d\lambda_1+2\mu_1\}\rightarrow\infty,
\end{align*}
where $u$ is a $H^1(D; \mathbb{R}^{d})$ solution of
 \begin{align}\label{main1}
\begin{cases}
\mathcal{L}_{\lambda, \mu}u:=\nabla\cdot(\mathbb{C}^0e(u))=0,\quad&\hbox{in}\ \Omega,\\
u|_{+}=u|_{-},&\hbox{on}\ \partial{D}_{1},\\
e(u)=0,&\hbox{in}~~D_{1},\\
\int_{\partial{D}_{1}}\frac{\partial u}{\partial \nu_0}\Big|_{+}\cdot\psi_{\alpha}=0,&\alpha=1,2,\cdots,\frac{d(d+1)}{2},\\
u=\varphi,&\hbox{on}\ \partial{D},
\end{cases}
\end{align}
where
\begin{align*}
\frac{\partial u}{\partial \nu_0}\Big|_{+}&:=(\mathbb{C}^0e(u))\vec{n}=\lambda(\nabla\cdot u)\vec{n}+\mu(\nabla u+(\nabla u)^T)\vec{n},
\end{align*}
and $\vec{n}$ is the unit outer normal of $D_{1}$. Here and throughout this paper the subscript $\pm$ indicates the limit from outside and inside the domain, respectively. The existence, uniqueness and regularity of weak solutions to (\ref{main1}) are proved in the Appendix of \cite{BLL}, where multiple inclusions case is studied. In particular, the $H^1$ weak solution to (\ref{main1}) is in $C^1(\overline{\Omega};\mathbb{R}^{d})\cap C^1(\overline{D}_{1};\mathbb{R}^{d})$. The solution is also the unique function which has the least energy in appropriate functional spaces, characterized by
$$I_\infty[u]=\min_{v\in\mathcal{A}}I_\infty[v],\qquad\,I_\infty[v]:=\frac{1}{2}\int_{\Omega}(\mathbb{C}^0e(v), e(v))dx,$$
where
\begin{equation}\label{def_A}
\mathcal{A}:=\left\{v\in H^1_\varphi(D;\mathbb{R}^{d}) ~\Big|~ e(v)=0\ \ \hbox{in}~~D_{1}\right\}.
\end{equation}

It is well known that for any open set $O$ and $u, v\in C^2(O)$,
\begin{align}\label{eu}
\int_O(\mathbb{C}^0e(u), e(v))dx=-\int_O\left(\mathcal{L}_{\lambda, \mu}u\right)\cdot v+\int_{\partial O}\frac{\partial u}{\partial \nu_0}\Big|_{+}\cdot v.
\end{align}
A calculation gives
\begin{align*}
\left(\mathcal{L_{\lambda, \mu}}u\right)_k=\mu\Delta u_k+(\lambda+\mu)\partial_{x_k}(\nabla\cdot u),\quad k=1, \cdots, d.
\end{align*}
We assume that for some $\delta_0>0$,
\begin{align}\label{delta}
\delta_0\leq \mu, d\lambda+2\mu\leq\frac{1}{\delta_0}.
\end{align}
It is clear that there exist two points $P_1\in\partial{D}_{1}$ and $P\in\partial{D}$, such that 
$$\mathrm{dist}(P, P_1)=\mathrm{dist}(D, \partial{D})=\varepsilon.$$
We use $\overline{P_{1}P}$ to denote the line segment connecting $P_{1}$ and $P$. Denote
$$\rho_{d}(\varepsilon)=
\begin{cases}
\sqrt{\varepsilon},\quad&\mbox{if}~~d=2,\\
\frac{1}{|\log\varepsilon|},&\mbox{if}~~d=3,\\
1,&\mbox{if}~~d\geq4.
\end{cases}$$

The first of our results concerns an upper bound of the gradient of solutions to (\ref{main1}). In brief, this result asserts that the blow up rate of $|\nabla u|$ is, respectively,
 $\epsilon^{-1/2}$ in dimension $d=2$,
$(\epsilon|\log\epsilon|)^{-1}$ in
dimension $d=3$,  and  $\epsilon^{-1}$  in dimension $d\ge 4$, which is exactly the same as the perfect conductivity problem, see e.g. \cite{BLY1}.
 
\begin{theorem}(Upper bound).\label{thm1.1}
Assume that $\Omega, D\subset\mathbb{R}^{d}$, $\varepsilon$ are defined in (\ref{Domain}), $\varphi\in C^2(\partial{D}; \mathbb{R}^{d})$. Let $u\in H^1(D; \mathbb{R}^{d})\cap C^1(\overline{\Omega}; \mathbb{R}^{d})$ be a solution to (\ref{main1}). Then for $0<\epsilon<1/2$,  we have
\begin{align}
|\nabla u(x)|&\leq
\frac{C\rho_{d}(\varepsilon)}{\varepsilon}\|\varphi\|_{C^2(\partial{D}; \mathbb{R}^{d})},\quad x\in\Omega,\label{gradient_in}
\end{align}
and
\begin{align}
|\nabla u(x)|&\leq C\|\varphi\|_{C^2(\partial{D}; \mathbb{R}^{d})},\quad x\in D_{1},\label{gradient2}
\end{align}
where $C$ depends only on $\kappa_{0},\delta_{0},d$, the $C^{2,\gamma}$ norm of $\partial{D}_{1}$ and $\partial{D}$, but not on $\varepsilon$.
\end{theorem}

\begin{remark}
Actually, for $d\geq2$, we have the following pointwise upper bound of $|\nabla{u}|$ in $\Omega$:
\begin{align}
|\nabla u(x)|&\leq\,C\left[
\frac{\rho_{d}(\varepsilon)}{\varepsilon+\mathrm{dist}^{2}(x,\overline{P_{1}P})}
+\left(\frac{ \mathrm{dist}(x, \overline{P_{1}P})}{\varepsilon+\mathrm{dist}^{2}(x,\overline{P_{1}P})}+1\right)\right]\|\varphi\|_{C^2(\partial{D}; \mathbb{R}^{d})}.\label{gradient1}
\end{align}
This shows that the right hand side archives its maximum at $\overline{P_{1}P}$, with value $\frac{C\rho_{d}(\varepsilon)}{\varepsilon}\|\varphi\|_{C^2(\partial{D}; \mathbb{R}^{d})}$ for $\varepsilon$ sufficiently small.
\end{remark}

In order to show that the blow-up rate of the gradients obtained in Theorem \ref{thm1.1} is  optimal, we need to investigate its lower bound. Denote $D_{1}^{*}:=\{~x\in\mathbb{R}^{d}~|~x+P_{1}\in{D}_{1}~\}$. Set $\Omega^*:=D\setminus \overline{D_{1}^{*}}$. Let $u_0^*$ be the solution of the boundary value problem:
\begin{align}\label{u00*}
\begin{cases}
  \mathcal{L}_{\lambda,\mu}u_0^*=0,\quad&
\hbox{in}\  \Omega^*,  \\
u_0^*=0,\ &\hbox{on}\ \partial{D}_{1}^*,\\
u_0^*=\varphi(x)-\varphi(\,P),&\hbox{on} \ \partial{D}.
\end{cases}
\end{align}
Define
\begin{align*}
b_{\alpha}^*:=\int_{\partial{D}_{1}^*}\frac{\partial u^*_0}{\partial \nu_0}\large|_{+}\cdot\psi_{\alpha},\quad\alpha=1,2,\cdots,\frac{d(d+1)}{2}.
\end{align*}
which is a functional of $\varphi$, playing an important role in the following establishment of lower bounds of $|\nabla u|$ on the segment $\overline{P_{1}P}$. 

\begin{theorem}(Lower bound).\label{thm1.2}
Under the assumption as in Theorem \ref{thm1.1}, let $u\in H^1(D; \mathbb{R}^{d})\cap C^1(\overline{\Omega}; \mathbb{R}^{d})$ be a solution to (\ref{main1}). Then 
\begin{itemize}
\item [$(i)$] for $d=2$, if there exists some integer $1\leq\,k_0\leq\,d$ such that $b_{k_0}^*\neq0$ and $\nabla_{x'}\varphi^{k_0}(\,P)=0$;
\item [$(ii)$] for $d=3$, if there exists some integer $1\leq\,k_0\leq\,d$ such that $b_{k_0}^*\neq0$; 
\item [$(iii)$] for $d\geq4$, if there exists some integer $1\leq\,k_0\leq\,d$ such that $b_{k_0}^*\neq0$ and $b_{\alpha}^*=0$ for all $\alpha\neq\,k_{0}$,
\end{itemize} 
then for sufficiently small $0<\varepsilon<1/2$,
$$\big|\nabla u(x)\big|\geq\frac{\rho_{d}(\varepsilon)}{C\varepsilon},\quad\,x\in\overline{P_{1}P},$$
where $C$ depends only on $\kappa_{0},\delta_{0},d$, the $C^{2,\gamma}$ norm of $\partial{D}_{1}$ and the $C^{2}$ norm of $\partial{D}$, but not on $\varepsilon$.
\end{theorem}

\begin{remark}
In Theorem \ref{thm1.2} we do not try to find the most general assumptions to guarantee blow-up occur, but instead give simple conditions (i)-(iii), which show, however, the essential role of the boundary data in this problem. Since $u_{0}$ is uniquely determined by \eqref{u00*} with given data $\varphi(x)-\varphi(\,P)$, Theorem \ref{thm1.2} shows that whether $|\nabla{u}|$ blows up or not totally depends only on the boundary data $\varphi(x)-\varphi(\,P)$. Furthermore, if the blow-up occurs, then from Theorem \ref{thm1.1} and \ref{thm1.2}, we know that it may occur only on the segment $\overline{P_{1}P}$.
\end{remark}

\begin{remark}
Theorem \ref{thm1.1} and \ref{thm1.2} give not only the upper bound but also a lower bound of the blow-up rate of the strain tensor in all dimensions, which shows the optimality of our estimates. Especially for the lower bound, new difficulties need to be overcome and a number of refined estimates are used in our proof. More important, a blow-up factor, totally depending on the given boundary data, is captured. 
\end{remark}

\begin{remark}
The strict convexity assumption on $\partial{D}$ and $\partial{D}_{1}$ in Theorem \ref{thm1.1} and \ref{thm1.2} can be extended to a weaker relative strict convexity assumption, see \eqref{h}--\eqref{h1} below.
\end{remark}

The organization of this paper is as follows. In Section \ref{sec_2} we first decompose the solution $u$ of \eqref{main1} as a linear combination of $u_{\alpha}$, $\alpha=1,2,\cdots,\frac{d(d+1)}{2}$, defined by \eqref{v-123} and \eqref{v-0} below, and then deduce the proof of Theorem \ref{thm1.1} to two aspects: the estimates of $|\nabla{u}_{\alpha}|$ and those of the coefficients $C^{\alpha}$ and $C^{\alpha}-\varphi^{\alpha}(0)$. In Section \ref{sec_3} we establish an upper bound of the gradient of solutions to a boundary problem of Lam\'e system on $\Omega$ with general Dirichlet boundary data in Theorem \ref{thm2.1}, of independent interest, and then obtain the estimates of $|\nabla{u}_{\alpha}|$ as a consequence of Theorem \ref{thm2.1}.  In Section \ref{sec_4} we present the estimates of the coefficients $C^{\alpha}$  and $C^{\alpha}-\varphi^{\alpha}(0)$. Theorem \ref{thm1.2} on the lower bound of $\nabla{u}$ on the segment $\overline{P_{1}P}$ is proved by studying the functional $b_{\alpha}^{*}$ of boundary data $\varphi$ in Section \ref{sec_5}. In the rest of the introduction we review some earlier results on interior gradient estimates for high contrast composites.

As mentioned before, Babu\u{s}ka, Andersson, Smith and Levin \cite{BASL} computationally analyzed the damage and fracture in composite materials and observed numerically that the size of the strain tensor remains bounded when the distance $\epsilon$, between two inclusions, tends to zero.
This was  proved by Li and Nirenberg in \cite{LN}.  Indeed such $\varepsilon$-independent gradient estimates was established there for solutions of divergence form second order elliptic systems, including linear systems of elasticity, with piecewise H\"{o}lder continuous coefficients in all dimensions. See Bonnetier and Vogelius \cite{BV} and Li and Vogelius  \cite{LV}  for corresponding results on divergence form elliptic equations.

The estimates in \cite{LN} and \cite{LV} depend on the ellipticity of the coefficients. If ellipticity constants are allowed to deteriorate, the situation is very different. Consider the simplied scalar model, also called as conductivity problem,
\begin{equation*}
\begin{cases}
\nabla\cdot\Big(a_{k}(x)\nabla{u}_{k}\Big)=0,&\mbox{in}~\Omega,\\
u_{k}=\varphi,&\mbox{on}~\partial\Omega,
\end{cases}
\end{equation*}
where $\Omega$ is a bounded open set of $\mathbb{R}^{d}$, $d\geq2$, containing two $\epsilon$-apart convex inclusions $D_{1}$ and $D_{2}$, $\varphi\in{C}^{2}(\partial\Omega)$ is given, and
$$a_{k}(x)=
\begin{cases}
k\in(0,\infty),&\mbox{in}~D_{1}\cup{D}_{2},\\
1,&\mbox{in}~\Omega\setminus\overline{D_{1}\cup{D}_{2}}.
\end{cases}
$$
When $k=\infty$,
the $L^\infty$-norm of $|\nabla u_\infty|$ for the solutions $u_{\infty}$ of  the following perfect conductivity problem
 \begin{align}\label{main_scalar}
\begin{cases}
\Delta u=0,\quad&\hbox{in}\ \Omega\setminus\overline{D_{1}\cup{D}_{2}},\\
u|_{+}=u|_{-},&\hbox{on}\ \partial{D}_{1}\cup\partial{D}_{2},\\
\nabla u=0,&\hbox{in}~~D_{1}\cup D_{2},\\
\int_{\partial{D}_{i}}\frac{\partial u}{\partial \vec{n}}\Big|_{+}=0,&i=1,2,\\
u=\varphi,&\hbox{on}\ \partial\Omega
\end{cases}
\end{align}
generally becomes unbounded as $\epsilon$ tends to $0$. There have been much more important progress on the interior gradient estimate of the solution of \eqref{main_scalar}, in contrast to the elasticity vector case. The blow up rate of $|\nabla u_\infty|$ is respectively
 $\epsilon^{-1/2}$ in dimension $d=2$,
$(\epsilon|\ln\epsilon|)^{-1}$ in
dimension $d=3$,  and  $\epsilon^{-1}$  in dimension $d\ge 4$. See
Bao, Li and Yin \cite{BLY1}, as well as Budiansky and Carrier \cite{BC},
  Markenscoff \cite{M}, Ammari, Kang and Lim \cite{AKL},
 Ammari, Kang, Lee, Lee and Lim \cite{AKLLL}, Yun  \cite{Y1,Y2} in $\mathbb{R}^{2}$, and Lim and Yun \cite{LY} in $\mathbb{R}^{3}$.
Further, more detailed, characterizations of the singular behavior of $\nabla{u}_{\infty}$ have been obtained by Ammari, Ciraolo, Kang, Lee and Yun \cite{ACKLY}, Ammari, Kang, Lee, Lim and Zribi \cite{AKLLiZ}, Bonnetier and Triki \cite{BT, BT2}, Gorb and Novikov \cite{GN} and
 Kang, Lim and Yun \cite{KLY, KLY2}. For more related works, see \cite{ABTV, ADKL, AGKL, AKKL, AKLLeZ, BLY2, BT, dong, DX, dongzhang, K1, K2, LLBY, LY} and the references therein.

\vspace{.5cm}

\section{Outline of the Proof of Theorem \ref{thm1.1}(Upper bound)}\label{sec_2}

We now describe our methods of proof. By a translation and rotation of the coordinates if necessary, we may assume without loss of generality that
$$P_1=\left(0, {\varepsilon}\right)\in\partial{D}_{1},\quad P=\left(0, 0\right)\in\partial{D}.$$
In order to prove Theorem 1.1, it suffices to consider the following problem, by replacing $u$ by $u-\varphi(0)$,
\begin{align}\label{main}
\begin{cases}
\mathcal{L}_{\lambda, \mu}u=0,\quad&\hbox{in}\ \Omega,\\
u|_{+}=u|_{-},&\hbox{on}\ \partial{D}_{1},\\
e(u)=0,&\hbox{in}~~D_{1},\\
\int_{\partial{D}_{1}}\frac{\partial u}{\partial \nu_0}\Big|_{+}\cdot\psi_\alpha=0,&\alpha=1, 2, \cdots,\frac{d(d+1)}{2},\\
u=\varphi(x)-\varphi(0),&\hbox{on}\ \partial{D}.
\end{cases}
\end{align}
By the third line of (\ref{main}) and the definition of $\Psi$, $u$ is a linear combination of $\{\psi_\alpha\}$ in $D_{1}$. Since it is clear that $\mathcal{L}_{\lambda,\mu}\xi=0$ in $\Omega$ and $\xi=0$ on $\partial\Omega$ imply that $\xi=0$ in $\Omega$, we decompose the solution of (\ref{main}), in the spire of \cite{BLY1},  as follows:
$$u=\sum_{\alpha=1}^{\frac{d(d+1)}{2}}C^{\alpha}\psi_{\alpha}-\varphi(0)=\sum_{\alpha=1}^{d}
(C^{\alpha}-\varphi^{\alpha}(0))\psi_{\alpha}+\sum_{\alpha=d+1}^{\frac{d(d+1)}{2}}C^{\alpha}\psi_{\alpha},\qquad\hbox{in}\ \overline{D}_{1},$$
for some constants $C^{\alpha}$, $\alpha=1,2,\cdots,\frac{d(d+1)}{2}$, (to be determined by the forth line in \eqref{main}) and
\begin{equation}\label{decompose}
u=\sum_{\alpha=1}^{d}
(C^{\alpha}-\varphi^{\alpha}(0))u_{\alpha}+\sum_{\alpha=d+1}^{\frac{d(d+1)}{2}}C^{\alpha}u_{\alpha}+u_0,\qquad\hbox{in}\ \Omega,
\end{equation}
where $u_\alpha\in C^1(\overline{\Omega}; \mathbb{R}^{d})\cap  C^2(\Omega; \mathbb{R}^{d})$, $\alpha=1, 2, \cdots,\frac{d(d+1)}{2}$, respectively, satisfy
\begin{align}\label{v-123}
\begin{cases}
\mathcal{L}_{\lambda, \mu}u_\alpha=0,\qquad&\hbox{in}\ \Omega,\\
u_\alpha=\psi_\alpha,&\hbox{on}\ \partial{D}_{1},\\
u_\alpha=0,&\hbox{on}~~\partial{D};
\end{cases}
\end{align}
 and $u_0\in C^1(\overline{\Omega}; \mathbb{R}^{d})\cap  C^2(\Omega; \mathbb{R}^{d})$  satisfies
\begin{align}\label{v-0}
\begin{cases}
\mathcal{L}_{\lambda, \mu}u_0=0,\quad&\hbox{in}\ \Omega,\\
u_0=0,&\hbox{on}\ \partial{D}_{1},\\
u_0=\varphi(x)-\varphi(0),&\hbox{on}~~\partial{D}.
\end{cases}
\end{align}
By the decomposition (\ref{decompose}), we write
\begin{align}\label{decomposition_nablau}
\nabla{u}=\sum_{\alpha=1}^{d}
(C^{\alpha}-\varphi^{\alpha}(0))\nabla{u}_{\alpha}+\sum_{\alpha=d+1}^{\frac{d(d+1)}{2}}C^{\alpha}\nabla{u}_{\alpha}+\nabla{u}_0, \quad\hbox{in}\ \Omega.
\end{align}

To estimate $|\nabla u|$, two ingredients are in order: (i) estimates of $|\nabla u_\alpha|$,  $\alpha=0,1,\cdots,\frac{d(d+1)}{2}$; (ii) estimates of $C^{\alpha}-\varphi^{\alpha}(0)$, $\alpha=1,\cdots,d$ and $C^{\alpha}$, $\alpha=1,\cdots,\frac{d(d+1)}{2}$.
Since the singular behavior of $\nabla u$ may occur only in the narrow region between  $D_{1}$ and $\partial{D}$, we are particularly interested in such narrow region. See Figure \ref{fig:1.2}.

\begin{figure}[t]
\includegraphics[width=3.0in]{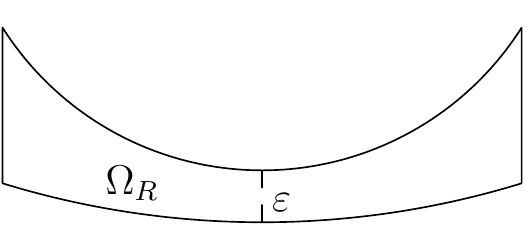}
\caption{\small The narrow region between $\partial{D}_{1}$ and $\partial{D}$.}
\label{fig:1.2}
\end{figure}

Fix a small constant $0<R<1$, independent of $\varepsilon$, such that the portions of $\partial{D}_{1}$ near $P_1$ and $\partial{D}$ near $P$ can be represented, respectively, by
$$x_{d}={\varepsilon}+h_1(x'),\quad\hbox{and}\quad x_{d}=h(x'),\quad\hbox{for}\ |x'|<2R.$$
Moreover, in view of the assumptions of $\partial{D}_{1}$  and $\partial{D}$, $h_{1}$ and $h$ satisfy
\begin{align}\label{h}
\varepsilon+h_1(x')>h(x'),\quad \hbox{for}\ |x'|<2R,\end{align}
\begin{align}\label{h2}h_1(0')=h(0')=0,\quad \nabla_{x'} h_1(0')=\nabla_{x'}h(0')=0,\end{align}
 \begin{align}\label{h1}
  \nabla^2_{x'}h_1(0'), \nabla^2_{x'}h(0')\geq \kappa_0I,\quad \nabla^2_{x'}(h_1-h)(0')\geq\kappa_1I,
 \end{align}
 and \begin{align}\label{h3}
\|h_1\|_{C^{2, \gamma}(\overline{B_{2R}(0')})}+\|h\|_{C^{2, \gamma}(\overline{B_{2R}(0')})}\leq \kappa_{2},
\end{align}
where $\kappa_0, \kappa_{1}$ and $\kappa_{2}$ are some positive constants. Throughout the paper, unless otherwise stated, we  use  $C$ to denote some positive constant, whose values may vary from line to line, which depend only on $\delta_0, \kappa_{0},\kappa_{1}$ and $\kappa_{2}$, but not on $\varepsilon$. Also, we call a constant having such dependence a {\it universal constant}.

For $0<r<2R$, we denote
 $$\Omega_r:=\left\{x=(x',x_{d})\in \mathbb{R}^{d}~\big|~ h(x')<x_{d}< \varepsilon+h_1(x'), ~|x'|<r \right\}.$$
The top and bottom boundaries of $\Omega_r$  are
 $$\Gamma_r^+=\{x\in \mathbb{R}^{d}\, |\, x_{d}=\varepsilon+h_1(x'), |x'|< r\},~ \Gamma_r^-=\{x\in \mathbb{R}^{d} \,|\, x_{d}=h(x'), |x'|<r\},$$respectively.

To estimate $|\nabla{u}_{\alpha}|$, we consider the following general boundary value problems:
\begin{align}\label{eq1.1}
\begin{cases}
\mathcal{L}_{\lambda,\mu}v:=\nabla\cdot(\mathbb{C}^0e(v))=0,\quad&
\hbox{in}\  \Omega,  \\
v=\psi(x),\quad &\hbox{on}\ \partial{D}_{1},\\
v=0, \quad&\hbox{on} \ \partial{D},
\end{cases}
\end{align}
where $ \psi(x)=(\psi^1(x), \psi^2(x),\cdots,\psi^{d}(x))^{T}\in C^2(\partial{D}_{1}; \mathbb{R}^{d})$ is given vector-valued functions. Locally piontwise gradient estimates for  problem (\ref{eq1.1}) is as follows:

\begin{theorem}\label{thm2.1}
Assume that hypotheses (\ref{h}) - (\ref{h3}) are satisfied, and let $v\in H^{1}(\Omega; \mathbb{R}^{d})$ be a weak solution of problem (\ref{eq1.1}). Then for $0<\varepsilon<1/2$, 
\begin{align}\label{mainestimate}
|\nabla v(x',x_{d})|
\leq&\, \frac{C}{\varepsilon+|x'|^2}\Big|\psi(x',\varepsilon+h_{1}(x'))\Big|+C\|\psi\|_{C^{2}(\partial{D}_{1}; \mathbb{R}^{d})},\quad\forall\,x\in \Omega_{R},
\end{align}
and
$$|\nabla v(x)|
\leq\,C\|\psi\|_{C^{2}(\partial{D}_{1}; \mathbb{R}^{d})}\quad\forall\,x\in \Omega\setminus\Omega_{R}.
$$
\end{theorem}

\begin{remark}
Theorem \ref{thm2.1} is of independent interest. We also can deal with more general case when $v=\phi(x)$ on $\partial{D}$, instead of the condition $v=0$ there. The proof of Theorem \ref{thm2.1} is given in Section \ref{sec_3}.
\end{remark}

Without loss of generality, we only need to prove Theorem \ref{thm1.1} for $\|\varphi\|_{C^{2}(\partial{D};\mathbb{R}^{d})}=1$, and for general case by considering $u/\|\varphi\|_{C^{2}(\partial{D};\mathbb{R}^{d})}$ if $\|\varphi\|_{C^{2}(\partial{D};\mathbb{R}^{d})}>0$. If $\varphi|_{\partial{D}}=0$, then $u\equiv0$. First, the estimates of $|\nabla{u}_{\alpha}|$ are some immediate consequences of Theorem \ref{thm2.1}, only taking $\psi=\psi_\alpha, \alpha=1,\cdots,\frac{d(d+1)}{2}$, respectively, or $\psi=\varphi(x)-\varphi(0)$ with minor modifications.

\begin{corollary}\label{corol2.2}
Under the hypotheses of Theorem 1.1 and with the normalization $\|\varphi\|_{C^{2}(\partial{D};\mathbb{R}^{d})}=1$. Then for $0<\varepsilon<1/2$,
\begin{align}
|\nabla u_\alpha(x)|\leq&\frac{C}{\varepsilon+|x'|^2},\quad \alpha=1,2,\cdots,d,\quad\forall~x\in \Omega_{R};\label{v12}\\
|\nabla u_{\alpha}(x)|\leq&\frac{C(\varepsilon+|x'|)}{\varepsilon+|x'|^2},\quad\alpha=d+1,\cdots,\frac{d(d+1)}{2},\quad\forall~x\in \Omega_{R};\label{v3}\\
|\nabla u_0(x)|\leq&\frac{C|\nabla\varphi(0)||x'|}{\varepsilon+|x'|^2}+C,\quad\forall~x\in \Omega_{R};\label{v0}
\end{align}
and
$$|\nabla u_{\alpha}(x)|
\leq\,C,\quad\alpha=0,1,2,\cdots,\frac{d(d+1)}{2},\quad\forall\,x\in \Omega\setminus\Omega_{R}.
$$
\end{corollary}

On the other hand, we need the following estimates on $C^{\alpha}$ and $|C^{\alpha}-\varphi^{\alpha}(0)|$. The proof is given in Section \ref{sec_4}.
\begin{prop}\label{prop1}
Under the hypotheses of Theorem 1.1 and with the normalization $\|\varphi\|_{C^{2}(\partial{D};\mathbb{R}^{d})}=1$. Then
\begin{align}&|C^\alpha|\leq C,\quad \alpha=1, 2, \cdots,\frac{d(d+1)}{2},\label{C3}
\end{align}
and
\begin{align}\label{C4}
 |C^{\alpha}-\varphi^{\alpha}(0)|\leq  C\rho_d(\varepsilon), \quad \alpha=1,2,\cdots,d.
\end{align}
\end{prop}

We are now in position to prove Theorem \ref{thm1.1}.

\begin{proof}[Proof of Theorem \ref{thm1.1}] 
Since
\begin{gather*}
\nabla u=\sum_{\alpha=d+1}^{\frac{d(d+1)}{2}}C^{\alpha}\nabla\psi_\alpha=\begin{pmatrix}0&C^{d+1}&C^{d+2}&\cdots&C^{2d-1}\\
-C^{d+1}&0&C^{2d}&\cdots &C^{3d-3}\\
-C^{d+2}&-C^{2d}&0& \ddots&\vdots\\
\vdots&\vdots& \ddots&\ddots &C^{\frac{d(d+1)}{2}}\\
-C^{2d-1}&-C^{3d-3}&\cdots&-C^{\frac{d(d+1)}{2}}&0
\end{pmatrix}\quad\hbox{in}~~D_{1},
\end{gather*}
The   estimate  (\ref{gradient2}) immediately follows from (\ref{C3}).

By (\ref{C4}), Corollary \ref{corol2.2} and Proposition \ref{prop1}, we have, for $x\in\Omega_{R}$,
\begin{align*}
|\nabla u(x)|&\leq\sum_{\alpha=1}^{d}
\left|C^{\alpha}-\varphi^{\alpha}(0)\right||\nabla{u}_{\alpha}|+\sum_{\alpha=d+1}^{\frac{d(d+1)}{2}}C^{\alpha}|\nabla{u}_{\alpha}|+|\nabla{u}_0|\\
&\leq\,C\left(
\frac{\rho_d(\varepsilon)}{\varepsilon+|x'|^{2}}
+\frac{|x'|}{\varepsilon+|x'|^{2}}+1\right).
\end{align*}
Thus, \eqref{gradient1} is proved, so \eqref{gradient_in}.
\end{proof}

To complete this section, we recall some properties of the tensor $\mathbb{C}$. For the isotropic elastic material, let
$$\mathbb{C}:=(C_{ijkl})=(\lambda\delta_{ij}\delta_{kl}+\mu(\delta_{ik}\delta_{jl}+\delta_{il}\delta_{jk})),\quad \mu>0,\quad d\lambda+2\mu>0.$$
The components $C_{ijkl}$ satisfy the following symmetric condition:
\begin{align}\label{symm}
C_{ijkl}=C_{klij}=C_{klji},\quad i,j,k,l=1,2,\cdots, d.
\end{align}
We will use the following notations:
\begin{align*}
(\mathbb{C}A)_{ij}=\sum_{k,l=1}^dC_{ijkl}A_{kl},\quad\hbox{and}\quad(A,B)\equiv A:B=\sum_{i,j=1}^dA_{ij}B_{ij},
\end{align*}
for every pair of $d\times d$ matrices $A=(A_{ij})$, $B=(B_{ij})$. Clearly,
$$(\mathbb{C}A, B)=(A, \mathbb{C}B).$$
If $A$ is symmetric, then, by the symmetry condition (\ref{symm}), we have that
$$(\mathbb{C}A, A)=C_{ijkl}A_{kl}A_{ij}=\lambda A_{ii}A_{kk}+2\mu A_{kj}A_{kj}.$$
Thus $\mathbb{C}$ satisfies the following ellipticity condition: For every $d\times d$ real symmetric matrix $\eta=(\eta_{ij})$,
\begin{align}\label{ellip}
\min\{2\mu, d\lambda+2\mu\}|\eta|^2\leq(\mathbb{C}\eta, \eta)\leq\max\{2\mu, d\lambda+2\mu\}|\eta|^2,
\end{align}
where $|\eta|^2=\sum_{ij}\eta_{ij}^2.$ In particular,
\begin{align}\label{2.15}
\min\{2\mu, d\lambda+2\mu\}|A+A^T|^2\leq(\mathbb{C}(A+A^T), (A+A^T)).
\end{align}

\vspace{1cm}

\section{Proof of Theorem \ref{thm2.1} and estimates of $|\nabla u_\alpha|$}\label{sec_3}

In this section, we first prove Theorem \ref{thm2.1}, then give some much finer estimates on $|\nabla u_\alpha|$, which will be useful for the establishment of the low bound estimates in Section \ref{sec_4} and Section \ref{sec_5}.

We decompose the solution of (\ref{eq1.1}) as follows:
\begin{align*}
v=v_1+v_2+\cdots+v_{d},
\end{align*}
where $v_l=(v_l^1, v_l^2,\cdots,v_{l}^{d})^{T}$, $l=1,2,\cdots,d$, with  $v_l^j=0$ for $j\neq l$, and $v_{l}$ satisfy the following boundary value problem, respectively,
\begin{align}\label{eq_v2.1}
\begin{cases}
  \mathcal{L}_{\lambda,\mu}v_{l}:=\nabla\cdot(\mathbb{C}^0e(v_{l}))=0,\quad&
\hbox{in}\  \Omega,  \\
v_{l}=( 0,\cdots,0,\psi^{l}, 0,\cdots,0)^{T},\ &\hbox{on}\ \partial{D}_{1},\\
v_{l}=0,&\hbox{on} \ \partial{D}.
\end{cases}
\end{align}
Then
\begin{equation}\label{equ_nablau}
\nabla{v}=\sum_{l=1}^{d}\nabla{v}_{l}.
\end{equation}

In order to estimate $|\nabla v_l|$ one by one, we first introduce a scalar auxiliary function $\bar{v}\in C^2(\mathbb{R}^n)$ such that $\bar{v}=1$ on $\partial{D}_{1}$, $\bar{v}=0$ on $\partial{D}$ and
$$\bar{v}(x)=\frac{x_{d}-h(x')}{\varepsilon+h_1(x')-h(x')},\quad\hbox{in}\ \Omega_{2R},$$
and
\begin{equation}\label{nabla_vbar_outside}
\|\bar{v}\|_{C^{2}(\Omega\setminus\Omega_{R/2})}\leq\,C.
\end{equation}
By a direct calculation, we obtain that for $k, j=1,\cdots, d-1$, and $x\in\Omega_{2R}$,
\begin{align}
|\partial_{x_{k}}\bar{v}(x)|\leq\frac{C|x'|}{\varepsilon+|x'|^2},\qquad ~~\frac{1}{C(\varepsilon+|x'|^2)}\leq|\partial_{x_{d}}\bar{v}(x)|\leq\frac{C}{\varepsilon+|x'|^2},\label{e2.4}
\end{align}
and
\begin{align}
|\partial_{x_{k}x_{j}}\bar{v}(x)|\leq\frac{C}{\varepsilon+|x'|^2}, \quad|\partial_{x_{k}x_{d}}\bar{v}(x)|\leq\frac{C|x'|}{(\varepsilon+|x'|^2)^2},\quad \partial_{x_{d}x_{d}}\bar{v}(x)=0.\label{ee2.4}
\end{align}

Extend $\psi\in{C}^{2}(\partial{D};\mathbb{R}^{d})$ to $\psi\in{C}^{2}(\overline{\Omega};\mathbb{R}^{d})$ such that $\|\psi^{l}\|_{C^{2}(\overline{\Omega\setminus\Omega_{R}})}\leq\,C\|\psi^{l}\|_{C^2(\partial D_1)}$, for $l=1,2,\cdots,d$. We can find $\rho\in{C}^{2}(\overline{\Omega})$ such that
\begin{equation}\label{cutoff_function_rho}
\begin{array}{l}
\displaystyle
0\leq\rho\leq1,~|\nabla\rho|\leq\,C,~\mbox{on}~\overline{\Omega},\\
\displaystyle\rho=1~\mbox{on}~\overline{\Omega}_{\frac{3}{2}R},~\mbox{and}~\rho=0~\mbox{on}~\overline{\Omega}\setminus\Omega_{2R}.
\end{array}
\end{equation}
Define
\begin{align}\label{equ_tildeu}
\tilde{v}_{l}(x)=( 0,\cdots,0,\left[\rho(x)\psi^{l}(x',\varepsilon+h_{1}(x'))+(1-\rho(x))\psi^{l}(x)\right]\bar{v}(x), 0,\cdots,0)^{T}\quad\mbox{in}~\Omega.
\end{align}
In particular, 
\begin{align}\label{equ_tildeu_in}
\tilde{v}_{l}(x)=( 0,\cdots,0,\psi^{l}(x',\varepsilon+h_{1}(x'))\bar{v}(x), 0,\cdots,0)^{T}\quad\mbox{in}~\Omega_{R},
\end{align}
and in view of \eqref{nabla_vbar_outside},
\begin{equation}\label{nabla_vtilde_outside}
\|\tilde{v}_{l}\|_{C^{2}(\Omega\setminus\Omega_{R/2})}\leq\,C\|\psi^{l}\|_{C^2(\partial D_1)}.
\end{equation}

Due to (\ref{e2.4}), and (\ref{ee2.4}), for $l=1,2,\cdots,d$, and $k,j=1,2,\cdots,d-1$, for $x\in\Omega_{R}$,
\begin{align}\label{eq1.7}
|\partial_{x_{k}}\tilde{v}_{l}(x)|\leq\frac{C|x'||\psi^{l}(x',\varepsilon+h_{1}(x'))|}{\varepsilon+|x'|^2}+C\|\nabla\psi^{l}\|_{L^{\infty}},
\end{align}
\begin{align}\label{eq1.7a}
\frac{|\psi^{l}(x',\varepsilon+h_{1}(x'))|}{C(\varepsilon+|x'|^2)}\leq|\partial_{x_{d}}\tilde{v}_{l}(x)|\leq\frac{C|\psi^{l}(x',\varepsilon+h_{1}(x'))|}{\varepsilon+|x'|^2};
\end{align}
and
\begin{align}
&|\partial_{x_{k}x_{j}}\tilde{v}_{l}(x)|\nonumber\\
&\leq\frac{C|\psi^{l}(x',\varepsilon+h_{1}(x'))|}{\varepsilon+|x'|^2}
+C\left(\frac{|x'|}{\varepsilon+|x'|^2}+1\right)\|\nabla\psi^{l}\|_{L^{\infty}}+C\|\nabla^{2}\psi^{l}\|_{L^{\infty}},\label{eq1.8}\\
&|\partial_{x_{k}x_{d}}\tilde{v}_{l}(x)|
\leq\frac{C|x'|}{(\varepsilon+|x'|^2)^2}|\psi^{l}(x',\varepsilon+h_{1}(x'))|+\frac{C}{\varepsilon+|x'|^2}\|\nabla\psi^{l}\|_{L^{\infty}},\label{eq1.9}\\
&\partial_{x_{d}x_{d}}\tilde{v}_{l}(x)=0.\label{eq1.10}
\end{align}
Here and throughout this section, for simplicity we use $\|\nabla\psi\|_{L^{\infty}}$ and $\|\nabla^{2}\psi\|_{L^{\infty}}$ to denote $\|\nabla\psi\|_{L^{\infty}(\partial{D}_{1})}$ and $\|\nabla^{2}\psi\|_{L^{\infty}(\partial{D_1})}$, respectively.

Let 
\begin{equation}\label{def_w}
w_l:=v_l-\tilde{v}_l,\qquad l=1,2,\cdots,d.
\end{equation}

\begin{lemma}\label{lem2.1}
Let $v_l\in H^1(\Omega; \mathbb{R}^{d})$ be a weak solution of (\ref{eq_v2.1}). Then
\begin{align}\label{lem2.2equ}
\int_{\Omega}|\nabla w_l|^2dx\leq C\|\psi^{l}\|_{C^{2}(\partial{D}_{1})},\qquad\,l=1,2,\cdots,d.
\end{align}
\end{lemma}

\begin{proof}
For simplicity, we denote
$$w:=w_{l},\quad\mbox{and}\quad \tilde{v}:=\tilde{v}_{l}.$$
Thus, $w$ satisfies
\begin{align}\label{eq2.6}
\begin{cases}
  \mathcal{L}_{\lambda,\mu}w=-\mathcal{L}_{\lambda,\mu}\tilde{v},&
\hbox{in}\  \Omega,  \\
w=0, \quad&\hbox{on} \ \partial\Omega.
\end{cases}
\end{align}
Multiplying the equation in (\ref{eq2.6}) by $w$ and applying integration by parts, we get
\begin{align}\label{integrationbypart}
\int_{\Omega}\left(\mathbb{C}^0e(w),e(w)\right)dx=\int_{\Omega}w\left(\mathcal{L}_{\lambda,\mu}\tilde{v}\right)dx.
\end{align}

By the Poincar\'{e} inequality,
\begin{align}\label{poincare_inequality}
\|w\|_{L^2(\Omega\setminus\Omega_R)}\leq\,C\|\nabla w\|_{L^2(\Omega\setminus\Omega_R)}.
\end{align}
Note that the above constant $C$ is independent of $\varepsilon$.
Using the Sobolev trace embedding theorem,
\begin{align}\label{trace}
\int\limits_{\scriptstyle |x'|={R},\atop\scriptstyle
h(x')<x_{d}<{\varepsilon}+h_1(x')\hfill}|w|dx\leq\ C \left(\int_{\Omega\setminus\Omega_{R}}|\nabla w|^2dx\right)^{\frac{1}{2}}.
\end{align}
According to (\ref{eq1.7}), we have
\begin{align}\label{nablax'tildeu}
&\int_{\Omega_{R}}|\nabla_{x'}\tilde{v}|^2dx\nonumber\\
&\leq C\int_{|x'|<R}(\varepsilon+h_1(x')-h(x'))\left(\frac{|x'|^2|\psi^{l}(x',\varepsilon+h_{1}(x'))|^{2}}{(\varepsilon+|x'|^2)^2}+\|\nabla\psi^{l}\|_{L^{\infty}}^{2}\right)dx'\nonumber\\
&\leq C\|\psi^{l}\|_{C^{1}(\partial{D}_{1})}^{2},
\end{align}
where $C$ depends only on $d$ and $\kappa_0$.

The first Korn's inequality together with \eqref{2.15}, \eqref{integrationbypart}, \eqref{nabla_vtilde_outside} and \eqref{poincare_inequality} implies
\begin{align}
\int_{\Omega}|\nabla w|^2dx\leq &\,2\int_{\Omega}|e(w)|^2dx\nonumber\\
\leq&\,C\left|\int_{\Omega_R}w(\mathcal{L}_{\lambda,\mu}\tilde{v})dx\right|+C\left|\int_{\Omega\setminus\Omega_R}w(\mathcal{L}_{\lambda,\mu}\tilde{v})dx\right|\nonumber\\
\leq&\,C\left|\int_{\Omega_R}w(\mathcal{L}_{\lambda,\mu}\tilde{v})dx\right|+C\|\psi^{l}\|_{C^2(\partial{D}_{1})}\int_{\Omega\setminus\Omega_R}|w|dx\nonumber\\
\leq&\,C\left|\int_{\Omega_R}w(\mathcal{L}_{\lambda,\mu}\tilde{v})dx\right|+C\|\psi^{l}\|_{C^2(\partial{D}_{1})}\left(\int_{\Omega\setminus\Omega_R}|\nabla w|^2\right)^{1/2},\nonumber
\end{align}
while, due to \eqref{eq1.10}, (\ref{trace}) and (\ref{nablax'tildeu}),
\begin{align}
&\left|\int_{\Omega_R}w(\mathcal{L}_{\lambda,\mu}\tilde{v})dx\right|
\leq C\sum_{k+l<2d}\left|\int_{\Omega_{R}}w\partial_{x_kx_l}\tilde{v}dx\right|\nonumber\\
\leq&\, C\int_{\Omega_{R}}|\nabla w||\nabla_{x'}\tilde{v}|dx+\int\limits_{\scriptstyle |x'|={R},\atop\scriptstyle
h(x')<x_{d}<\varepsilon+h_1(x')\hfill}C|\nabla_{x'} \tilde{v}||w|dx\nonumber \\
\leq&\,  C\left(\int_{\Omega_{R}}|\nabla w|^2dx\right)^{\frac{1}{2}}\left(\int_{\Omega_{R}}|\nabla_{x'}\tilde{v}|^2dx\right)^{\frac{1}{2}}+C\|\psi^{l}\|_{C^1(\partial{D}_{1})}\left(\int_{\Omega\setminus\Omega_R}|\nabla w|^2dx\right)^{\frac{1}{2}}\nonumber\\
\leq&\,  C\|\psi^{l}\|_{C^{1}(\partial{D}_{1})}\left(\int_{\Omega}|\nabla w|^2dx\right)^{\frac{1}{2}}.\nonumber
\end{align}
Therefore,
\begin{align*}
\int_{\Omega}|\nabla w|^2dx\leq  C\|\psi^{l}\|_{C^{2}(\partial{D}_{1})}\left(\int_{\Omega}|\nabla w|^2dx\right)^{\frac{1}{2}}.
\end{align*}
The proof of Lemma \ref{lem2.1} is completed.
\end{proof}

For convenience, we denote 
$$\delta(z):=\varepsilon+h_1(z')-h(z'), \quad\mbox{for}~~z=(z', z_{d})\in\Omega_{R}.$$  By (\ref{h1}), we have
\begin{align}
\frac{1}{C}(\varepsilon+|z'|^2)\leq\delta(z)\leq C(\varepsilon+|z'|^2).\nonumber
\end{align}
Set
$$\widehat{\Omega}_s(z):=\left\{~x\in\Omega_{2R} ~\big|~ |x'-z'|<s~\right\}, \quad\forall~ 0\leq{s}\leq R.$$
It follows from \eqref{equ_tildeu}  and (\ref{eq1.7})--(\ref{eq1.10}) that for $x\in\Omega_{R}$, $l=1,2,\cdots,d$,
\begin{align}\label{f}
|\mathcal{L}_{\lambda,\mu}\tilde{v}_{l}|\leq\, C|\nabla^2\tilde{v}_{l}|
\leq&\left(\frac{C}{\varepsilon+|x'|^2}+\frac{C|x'|}{(\varepsilon+|x'|^2)^2}\right)|\psi^{l}(x',\varepsilon+h_{1}(x'))|\nonumber\\
&+\frac{C}{\varepsilon+|x'|^2}\|\nabla\psi^{l}\|_{L^{\infty}}+C\|\nabla^{2}\psi^{l}\|_{L^{\infty}},
\end{align}
where $C$ is independent of $\varepsilon$.

\begin{lemma}\label{lem2.2}
For $\delta=\delta(z)\leq\,R$, $z\in\Omega_{R}$, and $l=1,2,\cdots,d$, \begin{align}\label{step2}
 \int_{\widehat{\Omega}_\delta(z)}|\nabla w_{l}|^2dx
 &\leq C\delta^{d+1}\left(|\psi^{l}(z',\varepsilon+h_{1}(z'))|^2+\delta(\|\psi^{l}\|_{C^2(\partial D_1)}^2+1)\right).
\end{align}
\end{lemma}

\begin{proof}
Still denote $w:=w_{l}$, and $\tilde{v}:=\tilde{v}_{l}$. For $0<t<s<1$,  let $\eta(x')$ be a smooth cutoff function satisfying $0\leq \eta(x')\leq1$, $\eta(x')=1$ if $|x'-z'|<t$, $\eta(x')=0$ if $|x'-z'|>s$ and $|\nabla\eta(x')|\leq\frac{2}{s-t}$. Multiplying $\eta^2w$ on both side of the equation in (\ref{eq2.6}) and applying integration by parts leads to
\begin{align}\label{lemma2.2-1}
\int_{\widehat{\Omega}_s(z)}(\mathbb{C}^0e(w), e(\eta^2w))dx=\int_{\widehat{\Omega}_s(z)}(\eta^2w)\mathcal{L}_{\lambda,\mu}\tilde{v}dx.
\end{align}
By the first Korn's inequality and the standard arguments, we have
\begin{align}\label{lemma2.2-2}
\int_{\widehat{\Omega}_s(z)}(\mathbb{C}^0e(w), e(\eta^2w))dx\geq\frac{1}{C}\int_{\widehat{\Omega}_s(z)}|\eta\nabla w|^2dx-C\int_{\widehat{\Omega}_s(z)}|\nabla\eta|^2|w|^2dx.
\end{align}
For the right hand side of (\ref{lemma2.2-1}), in view of H\"older inequality and Cauchy inequality,
\begin{align}
\left|\int_{\widehat{\Omega}_s(z)}(\eta^2w)\mathcal{L}_{\lambda,\mu}\tilde{v}dx\right|
&\leq\left(\int_{\widehat{\Omega}_s(z)}|w|^2dx\right)^{\frac{1}{2}}\left(\int_{\widehat{\Omega}_s(z)}|\mathcal{L}_{\lambda,\mu}\tilde{v}|^2dx\right)^{\frac{1}{2}}\nonumber\\
&\leq\frac{1}{(s-t)^2}\int_{\widehat{\Omega}_s(z)}|w|^2dx+(s-t)^2\int_{\widehat{\Omega}_s(z)}|\mathcal{L}_{\lambda,\mu}\tilde{v}|^2dx.\nonumber
\end{align}
This, together with (\ref{lemma2.2-1}) and  (\ref{lemma2.2-2}), implies that
\begin{align}\label{ww}
\int_{\widehat{\Omega}_t(z)}|\nabla w|^2dx\leq\frac{C}{(s-t)^2}\int_{\widehat{\Omega}_s(z)}|w|^2dx +C(s-t)^2\int_{\widehat{\Omega}_s(z)}|\mathcal{L}_{\lambda,\mu}\tilde{v}|^2dx.
\end{align}
We know that $w=0$ on $\Gamma_{R}^-$. By using (\ref{h})--\eqref{h3} and H\"{o}lder inequality, we obtain\begin{align}\label{w}
\int_{\widehat{\Omega}_s(z)}|w|^2dx&=\int_{\widehat{\Omega}_s(z)}\left|\int_{h(x')}^{x_{d}}\partial_{x_{d}}w(x', \xi)d\xi\right|^2dx\nonumber\\
&\leq \int_{\widehat{\Omega}_s(z)}(\varepsilon+h_1(x')-h(x'))\int_{h(x')}^{\varepsilon+h_1(x')}|\nabla w(x',\xi)|^2d\xi\,dx\nonumber\\
&\leq\,C\left(\varepsilon+(|z'|+s)^2\right)^{2}\int_{\widehat{\Omega}_s(z)}|\nabla w|^2dx.
\end{align}
It follows from (\ref{f}) and the mean value theorem that
\begin{align}
&\int_{\widehat{\Omega}_s(z)}|\mathcal{L}_{\lambda,\mu}\tilde{v}|^2dx\nonumber\\
\leq&\, |\psi^{l}(z',\varepsilon+h_{1}(z'))|^2\int_{\widehat{\Omega}_s(z)}\left(\frac{C}{\varepsilon+|x'|^2}+\frac{C|x'|}{(\varepsilon+|x'|^2)^2}\right)^2dx\nonumber\\
&+\|\nabla\psi^{l}\|_{L^\infty}^2\int_{\widehat{\Omega}_s(z)}\left(\frac{C}{\varepsilon+|x'|^2}+\frac{C|x'|}{(\varepsilon+|x'|^2)^2}\right)^2|x'-z'|^2dx\nonumber\\
&+\|\nabla\psi^{l}\|_{L^\infty}^2\int_{\widehat{\Omega}_s(z)}\left(\frac{C}{\varepsilon+|x'|^2}\right)^2dx+C\delta(z)s^{d-1}\|\nabla^2\psi^{l}\|_{L^\infty}^2\nonumber\\
\leq&\, C|\psi^{l}(z',\varepsilon+h_{1}(z'))|^2 \int_{|x'-z'|<s}\frac{dx'}{(\varepsilon+|x'|^2)^2}\nonumber\\
&+C\|\nabla\psi^{l}\|_{L^\infty}^2 \int_{|x'-z'|<s}\left(\frac{1}{\varepsilon+|x'|^2}+\frac{s^{2}}{(\varepsilon+|x'|^2)^2}\right)dx'+C\delta(z)s^{d-1}\|\nabla^2\psi^{l}\|_{L^\infty}^2.\label{ff}
\end{align}

{\bf Case 1.} For $0\leq |z'| \leq \sqrt{\varepsilon}$,  (i.e. $\varepsilon\leq\delta(z)\leq\,C\varepsilon$), and $0<t<s<\sqrt{\varepsilon}$, by means of (\ref{w}) and (\ref{ff}), we have
\begin{align}
&\int_{\widehat{\Omega}_s(z)}|w|^2dx\leq C\varepsilon^2\int_{\widehat{\Omega}_s(z)}|\nabla w|^2dx,\label{w21}
\end{align}
and
\begin{align}
&\int_{\widehat{\Omega}_s(z)}|\mathcal{L}_{\lambda,\mu}\tilde{v}|^2dx\nonumber\\
&\leq C|\psi^{l}(z',\varepsilon+h_{1}(z'))|^2\frac{s^{d-1}}{\varepsilon^{2}}+C\|\nabla\psi^{l}\|_{L^\infty}^2 \frac{s^{d-1}}{\varepsilon}+C\varepsilon\,s^{d-1}\|\nabla^2\psi^{l}\|_{L^\infty}^2.\label{f21}
\end{align}

Denote $$F(t):=\int_{\widehat{\Omega}_t(z)}|\nabla w|^2dx.$$ By (\ref{ww}), (\ref{w21}) and (\ref{f21}), for some universal constant $c_1>0$, we get for $0<t<s<\sqrt{\varepsilon}$,
\begin{align}\label{F}
F(t)\leq &\,\left(\frac{c_1\varepsilon}{s-t}\right)^2F(s)+C(s-t)^2s^{d-1} \cdot\nonumber\\
&\qquad\left(\frac{|\psi^{l}(z',\varepsilon+h_{1}(z'))|^2}{\varepsilon^2}+\frac{\|\nabla\psi^{l}\|_{L^\infty}^2}{\varepsilon}
+\varepsilon\|\nabla^2\psi^{l}\|_{L^\infty}^2\right).
\end{align}
Let  $t_i=\delta+2c_1i\varepsilon$, $i=0, 1, \cdots$ and $k=\left[\frac{1}{4c_1\sqrt{\varepsilon}}\right]+1$, then $$\frac{c_1\varepsilon}{t_{i+1}-t_i}=\frac{1}{2}.$$
Using (\ref{F}) with $s=t_{i+1}$ and $t=t_i$, we obtain
\begin{align*}
F(t_i)&\leq\frac{1}{4}F(t_{i+1})+C(i+2)^{d-1}\varepsilon^{d+1}\cdot\\
&\qquad\left(|\psi^{l}(z',\varepsilon+h_{1}(z'))|^2+\varepsilon(\|\nabla\psi^{l}\|_{L^\infty}^2+\|\nabla^2\psi^{l}\|_{L^\infty}^2)\right) , \quad i=0,1,2,\cdots, k.
\end{align*}
After $k$ iterations, making use of \eqref{lem2.2equ}, we have, for sufficiently small $\varepsilon$,
\begin{align*}
F(t_0)&\leq \big(\frac{1}{4}\big)^kF(t_k)+C\varepsilon^{d+1}\sum_{i=1}^k\big(\frac{1}{4}\big)^{i-1}(i+1)^{d-1}\cdot\\
&\qquad\left(|\psi^{l}(z',\varepsilon+h_{1}(z'))|^2+\varepsilon(\|\nabla\psi^{l}\|_{L^\infty}^2+\|\nabla^2\psi^{1}\|_{L^\infty}^2)\right)\\
&\leq \big(\frac{1}{4}\big)^kF(\sqrt{\varepsilon})+C\varepsilon^{d+1}\left(|\psi^{l}(z',\varepsilon+h_{1}(z'))|^2+\varepsilon(\|\nabla\psi^{l}\|_{L^\infty}^2+\|\nabla^2\psi^{l}\|_{L^\infty}^2)\right)\\
&\leq C\varepsilon^{d+1}\left(|\psi^{l}(z',\varepsilon+h_{1}(z'))|^2+\varepsilon\|\psi^{l}\|_{C^2(\partial{D}_{1})}^2\right),
\end{align*}
here we used the fact that $\big(\frac{1}{4}\big)^k\leq\big(\frac{1}{4}\big)^{\frac{1}{4c_{1}\sqrt{\varepsilon}}}\leq\,\varepsilon^{d+1}$ if $\varepsilon$ sufficiently small. This implies that for $0\leq |z'| \leq \sqrt{\varepsilon}$,
\begin{align*}
\|\nabla w\|_{L^2(\widehat{\Omega}_\delta(z))}^{2}\leq  C\varepsilon^{d+1}\left(|\psi^{l}(z',\varepsilon+h_{1}(z'))|^2+\varepsilon\|\psi^{l}\|_{C^2(\partial{D}_{1})}^2\right).
\end{align*}

{\bf Case 2.} For $\sqrt{\varepsilon}\leq |z'|<R$,  (i.e. $C|z'|^{2}\leq\delta(z)\leq(C+1)|z'|^{2}$), $0<t<s<\frac{2|z'|}{3}$, by using (\ref{w}) and (\ref{ff}) again, we have
\begin{align*}
&\int_{\widehat{\Omega}_s(z)}|w|^2dx\leq C|z'|^4\int_{\widehat{\Omega}_s(z)}|\nabla w|^2dx,\\
&\int_{\widehat{\Omega}_s(z)}|\mathcal{L}_{\lambda,\mu}\tilde{v}|^2dx\\
& \leq C|\psi^{l}(z',\varepsilon+h_{1}(z'))|^2\frac{s^{d-1}}{|z'|^4}+C\|\nabla\psi^{l}\|_{L^\infty}^2 \frac{s^{d-1}}{|z'|^2}+C|z'|^{2}s^{d-1}\|\nabla^2\psi^{l}\|_{L^\infty}^2.
\end{align*}
Thus, for $0<t<s<\frac{2|z'|}{3}$,
 \begin{align}\label{F1}
F(t)\leq&\, \left(\frac{c_2|z'|^2}{s-t}\right)^2F(s)+C(s-t)^2s^{d-1}\cdot\nonumber\\
&\qquad\left(\frac{|\psi^{l}(z',\varepsilon+h_{1}(z'))|^2}{|z'|^4}+\frac{\|\nabla\psi^{l}\|_{L^\infty}^2}{|z'|^2}
+|z'|^{2}\|\nabla^2\psi^{l}\|_{L^\infty}^2\right),
\end{align}
where $c_{2}$ is another universal constant. Taking the same iteration procedure as in Case 1, setting  $t_i=\delta+2c_2i|z'|^2$, $i=0, 1, \cdots$ and $k=\left[\frac{1}{4c_2|z'|}\right]+1$,
by (\ref{F1}) with $s=t_{i+1}$ and $t=t_i$, we have, for $i=0,1,2,\cdots, k$,
\begin{align*}
F(t_i)\leq&\,\frac{1}{4}F(t_{i+1})+C(i+2)^{d-1}|z'|^{2(d+1)}\cdot\\
&\qquad\left(|\psi^{l}(z',\varepsilon+h_{1}(z'))|^2+|z'|^{2}(\|\nabla\psi^{l}\|_{L^\infty}^2+\|\nabla^2\psi^{l}\|_{L^\infty}^2)\right).
\end{align*}
Similarly, after $k$ iterations, we have
\begin{align*}
F(t_0)\leq&\, \big(\frac{1}{4}\big)^kF(t_k)+C\sum_{i=1}^k\big(\frac{1}{4}\big)^{i-1}(i+1)^{d-1}|z'|^{2(d+1)}\cdot\\
&\qquad\left(|\psi^{l}(z',\varepsilon+h_{1}(z'))|^2+|z'|^{2}(\|\nabla\psi^{l}\|_{L^\infty}^2+\|\nabla^2\psi^{l}\|_{L^\infty}^2)\right)\\
\leq&\, \big(\frac{1}{4}\big)^kF(|z'|)\\
&+C|z'|^{2(d+1)}\left(|\psi^{l}(z',\varepsilon+h_{1}(z'))|^2+|z'|^{2}(\|\nabla\psi^{l}\|_{L^\infty}^2+\|\nabla^2\psi^{1}\|_{L^\infty}^2)\right)\\
\leq&\, C|z'|^{2(d+1)}\left(|\psi^{l}(z',\varepsilon+h_{1}(z'))|^2+|z'|^{2}\|\psi^{l}\|_{C^2(\partial{D}_{1})}^2\right),
\end{align*}
which implies that, for $ \sqrt{\varepsilon}\leq |z'|<R$,
\begin{align*}
\|\nabla w\|_{L^2(\widehat{\Omega}_\delta(z))}^2
&\leq C|z'|^{2(d+1)}\left(|\psi^{l}(z',\varepsilon+h_{1}(z'))|^2+|z'|^{2}\|\psi^{l}\|_{C^2(\partial{D}_{1})}^2\right).
\end{align*}
The proof of Lemma \ref{lem2.2} is completed.
\end{proof}

\begin{lemma}\label{lem2.3}
For $l=1,2,\cdots,d$,
\begin{align}\label{equa2.9}
|\nabla w_l(x)|\leq \frac{C|\psi^{l}(x',\varepsilon+h_{1}(x'))|}{\sqrt{\delta(x)}}+C\|\psi^{l}\|_{C^2(\partial{D}_{1})},\quad\forall~x\in\Omega_{R}.
\end{align}
Consequently, by (\ref{eq1.7}), \eqref{eq1.7a} and \eqref{def_w}, we have for sufficiently small $\varepsilon$ and $x\in\Omega_{R}$,
\begin{align}\label{estimate_vl}
\frac{|\psi^{l}(x',\varepsilon+h_{1}(x'))|}{C(\varepsilon+|x'|^2)}&\leq|\nabla v_l(x',x_{d})|\leq \frac{C|\psi^{l}(x',\varepsilon+h_{1}(x'))|}{\varepsilon+|x'|^2}+C\|\psi^{l}\|_{C^2(\partial{D}_{1})}.
\end{align}
\end{lemma}

\begin{proof}
Take $w:=w_{l}$  and $\tilde{v}:=\tilde{v}_{l}$ for simplicity. Given $z=(z',z_{d})\in\Omega_{R}$, making a change of variables
\begin{align*}
\begin{cases}
  x'-z'=\delta y', \\
x_{d}=\delta y_{d},
\end{cases}
\end{align*}
where $\delta=\delta(z)$. Define \begin{align*}
&\hat{h}_1(y'):=\frac{1}{\delta}\left(\varepsilon+h_1(\delta y'+z')\right),
\quad\hat{h}(y'):=\frac{1}{\delta}h(\delta y'+z').
\end{align*}
Then, the region $\widehat{\Omega}_\delta(z)$ becomes $Q_1$, where
$$Q_r=\{y\in \mathbb{R}^{d}\ |\ \hat{h}(y')<y_{d}<\hat{h}_1(y'),\ |y'|<r \},\quad 0<r\leq 1,$$ and the top and bottom boundaries of $Q_r$ become
$$\widehat{\Gamma}_r^+:=\{y\in \mathbb{R}^{d}\ |\ y_{d}=\hat{h}_1(y'),\ |y'|\leq r\}$$
and
$$\widehat{\Gamma}_r^-:=\{y\in \mathbb{R}^{d}\ |\ y_{d}=\hat{h}(y'),\ |y'|\leq r\},$$ respectively.
From (\ref{h})-\eqref{h3} and the definition of $\hat{h}_1$ and $\hat{h}$, it follows that
\begin{align*}
\hat{h}_1(0')-\hat{h}(0')=1,\end{align*}and for $|y'|<1$,
\begin{align*}
&|\nabla \hat{h}_1(y')|+|\nabla \hat{h}(y')|\leq C(\delta+|z'|),\quad
|\nabla^2 \hat{h}_1(y')|+|\nabla^2 \hat{h}(y')|\leq C\delta.
\end{align*}
Since $R$ is small, $\|\hat{h}_1\|_{C^{1, 1}(\overline{B_1(0')})}$ and $\|\hat{h}\|_{C^{1, 1}(\overline{B_1(0')})}$ are small and  $Q_1$ is approximately a unit square (or a cylinder-shaped domain) as far as applications of the Sobolev embedding theorems and
classical $L^p$ estimates for elliptic systems are concerned.

Let $$\hat{v}(y', y_{d}):=\tilde{v}(\delta y'+z', \delta y_{d}),\qquad \hat{w}(y', y_{d}):=w(\delta y'+z', \delta y_{d}).$$ Thus, $\hat{w}(y)$ satisfies
 \begin{align*}
\begin{cases}
  \mathcal{L}_{\lambda,\mu}\hat{w}=-\mathcal{L}_{\lambda,\mu}\hat{v}\quad
&\hbox{in}\  Q_1,  \\
\hat{w}=0, \quad&\hbox{on} \ \widehat{\Gamma}_1^\pm.
\end{cases}
\end{align*}

In view of $\hat{w}=0$ on the upper and lower boundaries of $Q_1$, we have, by Poincar\'{e} inequality, that
$$\|\hat{w}\|_{H^1(Q_1)}\leq C\|\nabla \hat{w}\|_{L^2(Q_1)}.$$
 Using the Sobolev embedding theorem and classical $W^{2, p}$ estimates for elliptic systems (see e.g. \cite{AD2}, or theorem 2.5 in \cite{G}), we have, for some $p>n$,
 $$\|\nabla \hat{w}\|_{L^\infty(Q_{1/2})}\leq C\|\hat{w}\|_{W^{2, p}(Q_{1/2})}\leq C\left(\|\nabla \hat{w}\|_{L^2(Q_1)}+\|\mathcal{L}_{\lambda,\mu}\hat{v}\|_{L^\infty(Q_1)}\right).$$
Since $$\|\nabla \hat{w}\|_{L^\infty(Q_{1/2})}=\delta\|\nabla {w}\|_{L^\infty(\widehat{\Omega}_{\delta/2}(z))},\quad\|\nabla\hat{w}\|_{L^2(Q_1)}=\delta^{1-\frac{d}{2}}\|\nabla w\|_{L^2(\widehat{\Omega}_\delta(z))}$$and
$$\|\mathcal{L}_{\lambda,\mu}\hat{v}\|_{L^\infty(Q_1)}=\delta^{2}\|\mathcal{L}_{\lambda,\mu}\tilde{v}\|_{L^\infty(\widehat{\Omega}_\delta(z))}$$
Tracking back to $w$ through the transforms, we have
\begin{align}\label{Linfty_estimate}
\|\nabla w\|_{L^\infty(\widehat{\Omega}_{\delta/2}(z))}\leq\frac{C}{\delta}\left(\delta^{1-\frac{d}{2}}\|\nabla w\|_{L^2(\widehat{\Omega}_\delta(z))}+\delta^2\|\mathcal{L}_{\lambda,\mu}\tilde{v}\|_{L^\infty(\widehat{\Omega}_\delta(z))}\right).
\end{align}

By  (\ref{f}) and (\ref{step2}), we have
\begin{align*}
\delta^{-\frac{d}{2}}\|\nabla w\|_{L^2(\widehat{\Omega}_\delta(z))}
&\leq \frac{C}{\sqrt{\delta}}|\psi^{l}(z',\varepsilon+h_{1}(z'))|
+C\|\psi^{l}\|_{C^2(\partial D_1)},\end{align*}
and
\begin{align*}
\delta \|\mathcal{L}_{\lambda,\mu}\tilde{v}\|_{L^\infty(\widehat{\Omega}_\delta(z))}\leq \frac{C}{\sqrt{\delta}}|\psi^{l}(z',\varepsilon+h_{1}(z'))|
+C(\|\nabla\psi^{1}\|_{L^\infty}+\|\nabla^2\psi^{l}\|_{L^\infty}).
\end{align*}
Plugging these estimates above into \eqref{Linfty_estimate} yields \eqref{equa2.9}. The proof of Lemma \ref{lem2.3} is finished.
\end{proof}

\begin{proof}[Proof of Theorem \ref{thm2.1}]
By using \eqref{estimate_vl} and the decomposition of $\nabla{u}$, \eqref{equ_nablau}, 
\begin{align*}
|\nabla v(x)|&\leq\sum_{l=1}^{d}|\nabla{v}_{l}|\leq \frac{C|\psi(x',\varepsilon+h_{1}(x'))|}{\varepsilon+|x'|^2}+C\|\psi\|_{C^2(\partial D_1)},\quad\,x\in\Omega_R.
\end{align*}
Note that  for any $x\in\Omega\setminus\Omega_{R}$, by using the standard interior estimates and boundary estimates for elliptic systems \eqref{eq1.1} (see Agmon et al. \cite{AD1} and \cite{AD2}), we have
 $$\|\nabla v\|_{L^\infty(\Omega\setminus \Omega_{R})}\leq C\|\psi\|_{C^2(\partial D_1)}.$$  The proof of Theorem \ref{thm2.1} is completed.
\end{proof}

The following finer estimates in $\Omega_{R}$ will be useful in Section \ref{sec_4} and Section \ref{sec_5}. We assume that $\|\varphi\|_{C^2(\partial{D}; \mathbb{R}^{d})}=1$ without loss of generality. For problem (\ref{v-123}), taking
$$\psi=\psi_\alpha,\quad\mbox{and}\quad \tilde{u}_\alpha:=\bar{v}\psi_\alpha,~ \alpha=1,\cdots,d$$
 in the proof of Lemma 3.1-Lemma 3.3, respectively, we have

\begin{corollary}\label{corol3.4}
For $\alpha=1,2,\cdots,d$,
\begin{align}\label{cor3.4-1}
|\nabla (u_\alpha-\tilde{u}_\alpha)(x)|&\leq\frac{C}{\sqrt{\delta(x)}},\qquad\forall~~x\in\Omega_{R}.
\end{align}
Consequently, by the definition of $\tilde{u}_\alpha$ and \eqref{e2.4}, we have, for $\alpha=1,\cdots,d$,
\begin{align}\label{v-alpha1}
|\nabla_{x'} u_\alpha(x)|&\leq\frac{C}{\sqrt{\delta(x)}},\qquad\forall~~x\in\Omega_{R},
\end{align}and
\begin{align}\label{v-alpha2}
 \frac{1}{C\delta(x)}\leq|\partial_{x_d} u_\alpha(x)|\leq\frac{C}{\delta(x)},\quad x\in\Omega_R.
\end{align}
\end{corollary}

\begin{proof}
According to the definition of $\tilde{u}_\alpha$ and \eqref{e2.4}, we have
\begin{align*}
|\nabla_{x'}\tilde{u}_\alpha(x)|&\leq\frac{C|x'|}{\varepsilon+|x'|^2}, \quad x\in\Omega_{R},\\
~~\frac{1}{C(\varepsilon+|x'|^2)}&\leq|\partial_{x_{d}}\tilde{u}_\alpha(x)|\leq\frac{C}{\varepsilon+|x'|^2},~~x\in\Omega_R;
\end{align*}
and
\begin{align*}
|\mathcal{L}_{\lambda,\mu}\tilde{u}_\alpha(x)|\leq&\, C\sum_{k+l<2d}|\partial_{x_{k}x_{l}}\tilde{u}_\alpha(x)|\\
\leq&\left(\frac{C}{\varepsilon+|x'|^2}+\frac{C|x'|}{(\varepsilon+|x'|^2)^2}\right)|\psi^{l}(x',\varepsilon+h_{1}(x'))| ,~~x\in\Omega_R.
\end{align*}
Clearly, (\ref{cor3.4-1}) follows from the proof of Lemma \ref{lem2.1}--Lemma \ref{lem2.3}. 
\end{proof}

For problem (\ref{v-0}), we
decompose the solution $u_0$ as
$$u_0=u_{01}+u_{02}+\cdots+u_{0d},$$ where $u_{0l}$, $l=1,2,\cdots,d$, satisfy, respectively,
 \begin{align}\label{v01}
\begin{cases}
  \mathcal{L}_{\lambda,\mu}u_{0l}=0,\quad&
\hbox{in}\  \Omega,  \\
u_{0l}=0,\ &\hbox{on}\ \partial{D}_{1},\\
u_{0l}=(0,\cdots,0,\varphi^{l}(x)-\varphi^{l}(0),0,\cdots,0)^T, &\hbox{on} \ \partial{D}.
\end{cases}
\end{align}

Similar as \eqref{equ_tildeu}, we define  
\begin{align}\label{equ_tildeu0}
\tilde{u}_{0l}(x):=( 0,\cdots,0,&\big[\,\rho(x)(\varphi^{l}(x', h(x'))-\varphi^{l}(0))\nonumber\\
&+(1-\rho(x))(\varphi^{l}(x)-\varphi^{l}(0))\big](1-\bar{v})(x), 0,\cdots,0)^{T},\quad\forall~x\in\Omega.
\end{align}
where $\rho\in{C}^{2}(\overline{\Omega})$ is a cutoff function satisfying \eqref{cutoff_function_rho} as before. In particular,
$$\tilde{u}_{0l}=(0,\cdots,0,(\varphi^{l}(x', h(x'))-\varphi^{l}(0))(1-\bar{v})(x),0,\cdots,0)^T,\quad\forall~~x\in\Omega_{R}.$$ 
Adapting the proofs of Lemma \ref{lem2.1}--Lemma \ref{lem2.3} to the equation (\ref{v01}), we obtain the following corollary.

\begin{corollary}\label{corol3.5} For $l=1,2,\cdots,d$,
\begin{align}\label{step3-3}
|\nabla(u_{0l}-\tilde{u}_{0l})(x)|
\leq C\|\varphi^{l}\|_{C^2(\partial{D})}, \qquad~x\in\Omega_R.
\end{align}
Consequently,
\begin{align}\label{step3-4}
|\nabla_{x'}u_{0l}(x)|&\leq\,C\|\varphi^{l}\|_{C^2(\partial D)}, \quad~x\in\Omega_R,
\end{align}
and
\begin{align}\label{step3-5}
\frac{|\varphi^{l}(x',h(x'))-\varphi^{l}(0)|}{C\delta(x)}\leq|\partial_{x_{d}}u_{0l}(x)|\leq
\frac{C|\nabla_{x'}\varphi^{l}(0)||x'|}{\delta(x)}+C\|\varphi^{l}\|_{C^2(\partial D)},~x\in\Omega_R.
\end{align}
\end{corollary}

\begin{proof} 
For (\ref{v01}), it is clear from \eqref{equ_tildeu0} that $\tilde{u}_{0l}=u_{0l}=0$ on $\partial{D}_{1}$, $\tilde{u}_{0l}=u_{0l}$ on $\partial{D}$. Note that $\tilde{u}_{0l}^{k}=0$, if $k\neq\,l$, and for $x\in\Omega_{R}$,
\begin{align*}
\nabla_{x'}\tilde{u}_{0l}^{l}=&\,-\left(\varphi^{l}(x', h(x'))-\varphi^{l}(0)\right)\nabla_{x'}\bar{v}(x)\\
&+\Big[\nabla_{x'}\varphi^{l}(x',h(x'))+\partial_{x_{d}}\varphi^{l}(x',h(x'))\nabla_{x'}h(x')\Big](1-\bar{v})(x),\\
\partial_{x_d}\tilde{u}_{0l}^{l}=&\,-\left(\varphi^{l}(x', h(x'))-\varphi^{l}(0)\right)\partial_{x_{d}}\bar{v}(x).
\end{align*}
By the Taylor expansion and (\ref{h2})-(\ref{h1}), 
\begin{align}\label{taylor}
\varphi^{l}(x',h(x'))=&\,\varphi^{l}(0)+\nabla_{x'}\varphi^{l}(0)x'\nonumber\\
&+\frac{1}{2}{x'}^T\Big[\nabla^2_{x'}\varphi^{l}(0)+\partial_{x_{d}}\varphi^{l}(0)\nabla^2_{x'}h(0')\Big]x'+O(|x'|^{2+\gamma}).
\end{align}
Hence, using (\ref{e2.4}), we have
\begin{align}\label{step3-1}
|\nabla_{x'}\tilde{u}_{0l}(x)|&\leq\frac{C|\nabla_{x'}\varphi^{l}(0)||x'|^2}{\varepsilon+|x'|^2}+C\|\varphi^{l}\|_{C^2(\partial D)}\leq\,C\|\varphi^{l}\|_{C^2(\partial D)}, ~x\in\Omega_R,
\end{align}
and
\begin{align}\label{step3-2}
\frac{|\varphi^{l}(x',h(x'))-\varphi^{l}(0)|}{C(\varepsilon+|x'|^2)}\leq|\partial_{x_{d}}\tilde{u}_{0l}(x)|\leq
\frac{C|\nabla_{x'}\varphi^l(0)||x'|}{\varepsilon+|x'|^2}+C\|\varphi^{l}\|_{C^2(\partial D)},~x\in\Omega_R.
\end{align}

Adapting the proof of Lemma \ref{lem2.1}--Lemma \ref{lem2.2} and using (\ref{taylor}), we obtain
\begin{align*}
|\nabla(u_{0l}-\tilde{u}_{0l})(x)|\leq&
\frac{C|\nabla_{x'}\varphi^{l}(0)||x'|}{\sqrt{\varepsilon+|x'|^2}}+C\|\varphi^{l}\|_{C^2(\partial{D})}
\leq\,C\|\varphi^{l}\|_{C^2(\partial{D})}, \quad~x\in\Omega_R,
\end{align*}
which, together with (\ref{step3-1}) and (\ref{step3-2}), implies that (\ref{step3-4}) and (\ref{step3-5}).
\end{proof}

\vspace{1cm}

\section{Proof of Proposition \ref{prop1} and estimates of $C^\alpha$}\label{sec_4}

In this Section, we are devoted to prove Proposition \ref{prop1} under the normalization $\|\varphi\|_{C^2(\partial{D}; \mathbb{R}^{d})}=1$. 

Denote
\begin{align*}
a_{\alpha\beta}:=-\int_{\partial{D}_{1}}\frac{\partial u_\alpha}{\partial \nu_0}\large\Big|_{+}\cdot\psi_{\beta},\quad b_{\beta}:=\int_{\partial{D}_{1}}\frac{\partial u_0}{\partial \nu_0}\large\Big|_{+}\cdot\psi_{\beta},\quad \alpha, \beta=1,2,\cdots, \frac{d(d+1)}{2}.
\end{align*}
Multiplying the first line of (\ref{v-123}) and (\ref{v-0}), by $u_\beta$, respectively, and applying integration by parts over $\Omega$ leads to
$$a_{\alpha\beta}=\int_{\Omega}(\mathbb{C}^0e(u_\alpha), e(u_\beta))dx,\quad b_{\beta}=-\int_{\Omega}(\mathbb{C}^0e(u_0), e(u_\beta))dx.$$
By \eqref{decomposition_nablau} and the linearity of $e(u)$,
\begin{align*}
e(u)=\sum_{\alpha=1}^{d}(C^{\alpha}-\varphi^{\alpha}(0))e(u_\alpha)+\sum_{\alpha=d+1}^{\frac{d(d+1)}{2}}C^{\alpha}e(u_{\alpha})+e(u_0),\quad\hbox{in}\ \Omega.
\end{align*}
Then, it follows from the forth line of (\ref{main}) that for $\beta=1,2,\cdots,\frac{d(d+1)}{2}$,
\begin{align}\label{eq5.2}
\sum_{\alpha=1}^{d}(C^{\alpha}-\varphi^{\alpha}(0))a_{\alpha\beta}+\sum_{\alpha=d+1}^{\frac{d(d+1)}{2}}C^{\alpha}a_{\alpha\beta}
=b_{\beta}.
\end{align}

Denote  
$$X^1=(C^1-\varphi^1(0),\cdots,C^d-\varphi^d(0))^T,\quad\, X^2=(C^{d+1},\cdots,C^{\frac{d(d+1)}{2}})^T,$$
$$P^1=(b_1,\cdots, b_d)^T,\quad\, P^2=(b_{d+1},\cdots, b_{\frac{d(d+1)}{2}})^T,$$ 
and
\begin{gather*}A=\begin{pmatrix} a_{11}&\cdots&a_{1d} \\  \vdots&\ddots&\vdots\\ a_{d1}&\cdots&a_{dd}\end{pmatrix},\quad
B=\begin{pmatrix} a_{1~d+1}&\cdots&a_{1~\frac{d(d+1)}{2}}\\ \vdots&\ddots&\vdots\\ a_{d~d+1}&\cdots&a_{d~\frac{d(d+1)}{2}} \end{pmatrix}, \\
 D=\begin{pmatrix} a_{d+1~d+1}&\cdots&a_{d+1~\frac{d(d+1)}{2}}\\ \vdots&\ddots&\vdots\\ a_{\frac{d(d+1)}{2}~d+1}&\cdots&a_{\frac{d(d+1)}{2}~\frac{d(d+1)}{2}} \end{pmatrix}.
\end{gather*}
Thus, by using the symmetry property of $a_{\alpha\beta}$, \eqref{eq5.2} can be rewritten as
\begin{gather}\label{eq5.3}\begin{pmatrix}A&B\\B^T&D\end{pmatrix}
\begin{pmatrix}X^1\\X^2\end{pmatrix}=\begin{pmatrix}P^1\\P^2\end{pmatrix}
\end{gather}

\begin{lemma}\label{lemma4.1}
There exists a positive universal constant $C$, independent of $\varepsilon$, such that
\begin{align}\label{lemma4.1-1}
\sum_{\alpha,\beta=1}^{\frac{d(d+1)}{2}}a_{\alpha\beta}\xi_\alpha\xi_\beta\geq\frac{1}{C},\quad\forall~ \xi\in \mathbb{R}^{\frac{d(d+1)}{2}},~ |\xi|=1.
\end{align}
\end{lemma}

\begin{proof}
To emphasize the dependence on $\varepsilon$, we use $\Omega_\varepsilon:=D\setminus\overline{D_1}$ and $u_{\alpha}^{\varepsilon}$ to denote the corresponding solution of (\ref{v-123}) with $\alpha=1,\cdots, d$. For $\xi\in \mathbb{R}^{\frac{d(d+1)}{2}}$ with $|\xi|=1$, using (\ref{delta}), we have
\begin{align*}
\sum_{\alpha,\beta=1}^{\frac{d(d+1)}{2}}a_{\alpha\beta}\xi_\alpha\xi_\beta&=\int_{\Omega}\left(\mathbb{C}^0e\Big(\sum_{\alpha=1}^{\frac{d(d+1)}{2}}\xi_\alpha u_\alpha^{\varepsilon}\Big), e\Big(\sum_{\alpha=1}^{\frac{d(d+1)}{2}}\xi_\beta u_\beta^{\varepsilon}\Big)\right)dx\\
&\geq\frac{1}{C}\int_\Omega\left|e\Big(\sum_{\alpha=1}^{\frac{d(d+1)}{2}}\xi_\alpha u_\alpha^{\varepsilon}\Big)\right|^2dx.
\end{align*}
We claim that there exists a constant $C>0$, independent of $\varepsilon$, such that
\begin{align*}
\int_\Omega\left|e\Big(\sum_{\alpha=1}^{\frac{d(d+1)}{2}}\xi_\alpha u_\alpha^{\varepsilon}\Big)\right|^2dx\geq\frac{1}{C},\quad\forall\ \xi\in \mathbb{R}^{\frac{d(d+1)}{2}},~|\xi|=1.
\end{align*}
Indeed, if not, then there exist $\varepsilon_i\rightarrow0^+$, $|\xi^i|=1$, such that
\begin{align} \label{case-3}
\int_\Omega\left|e\Big(\sum_{\alpha=1}^{\frac{d(d+1)}{2}}\xi^i_\alpha u_\alpha^{\varepsilon_{i}}\Big)\right|^2dx\rightarrow 0,\quad\mbox{as}~i\to\infty.
\end{align}

Here and in the following proof, we use the notations $D_1^*:=\{~x\in\mathbb{R}^{d} ~|~x+(0',\varepsilon)\in\,D_1~\}$, $\Omega^*:=D\setminus\overline{D_1^*} $. Since $u_{\alpha}^{\varepsilon_{i}}=0$ on $\partial D$, it follows from the second Korn's inequality (see theorem 2.5 in \cite{O}) that there exists a constant $C$, independent of $\varepsilon_{i}$, such that
\begin{align*}
\|u_{\alpha}^{\varepsilon_{i}}\|_{H^1(\Omega_\varepsilon\setminus B_{\bar{r}};\mathbb{R}^{d})}\leq C,
\end{align*}
where $\bar{r}>0$ is fixed. Then there exists a subsequence, we still denote $\{u_{\alpha}^{\varepsilon_{i}}\}$, such that
\begin{align*}
u_{\alpha}^{\varepsilon_{i}}\rightharpoonup \bar{u}_\alpha,\quad \hbox{in}\ H^1(\Omega_\varepsilon\setminus B_{\bar{r}};\mathbb{R}^{d}),\quad\hbox{as}~~\, i\rightarrow\infty.
\end{align*}
By (\ref{case-3}), there exists $\bar{\xi}$ such that
\begin{align*}
\xi^i\rightarrow\bar{\xi},\quad\hbox{as}\ i\rightarrow\infty,\quad\hbox{with}\ |\bar{\xi}|=1,
\end{align*}
and
\begin{align*}
\int_{\Omega^*}\left|e\Big(\sum_{\alpha=1}^{\frac{d(d+1)}{2}}\bar{\xi}_\alpha \bar{u}_\alpha\Big)\right|^2dx=0.
\end{align*}
This implies that
\begin{align*}
e\Big(\sum_{\alpha=1}^{\frac{d(d+1)}{2}}\bar{\xi}_\alpha \bar{u}_\alpha\Big)=0,\quad\hbox{in}\ \Omega^*.
\end{align*}
That means that $\sum_{\alpha=1}^{\frac{d(d+1)}{2}}\bar{\xi}_\alpha \bar{u}_\alpha\in\Psi$ in $\Omega^*$. Hence, there exist some constants $c_\beta$, $\beta=1,2,\cdots, \frac{d(d+1)}{2}$, such that
\begin{align*}
\sum_{\alpha=1}^{\frac{d(d+1)}{2}}\bar{\xi}_\alpha \bar{u}_\alpha=\sum_{\beta=1}^{\frac{d(d+1)}{2}}c_\beta\psi_\beta,\quad\hbox{in}\ \Omega^*.
\end{align*}
Since $\sum_{\beta=1}^{\frac{d(d+1)}{2}}c_\beta\psi_\beta=0$ on $\partial D$, it follows from lemma 6.1 in \cite{BLL2} that $c_\beta=0,\ \beta=1,\cdots, ^{\frac{d(d+1)}{2}}.$
Thus,
\begin{align*}
\sum_{\alpha=1}^{\frac{d(d+1)}{2}}\bar{\xi}_\alpha \bar{u}_\alpha=0,\quad\hbox{in}\ \Omega^*.
\end{align*}
Restricted on $\partial D_1^*$, it says that $\sum_{\alpha=1}^{\frac{d(d+1)}{2}}\bar{\xi}_\alpha \psi_\alpha=0$ on $\partial D_1^*$. This yields, using again lemma 6.1 in \cite{BLL2}, $\bar{\xi}_\alpha=0, \alpha=1,\cdots, d$, which contradicts with $|\bar{\xi}|=1$. 
\end{proof}

\begin{lemma}\label{lemma4.2} For $d\geq2$, we have
\begin{align}
\frac{1}{C\rho_d(\varepsilon)}\leq a_{\alpha\alpha}&\leq \frac{C}{\rho_d(\varepsilon)},\quad \alpha=1, \cdots, d;\label{lem5.8-1}\\
\frac{1}{C}\leq a_{\alpha\alpha}&\leq C,\quad \alpha=d+1, \cdots, \frac{d(d+1)}{2};\label{lem5.8-2}\\
a_{\alpha\beta}&\leq C,\quad \alpha=1,2,\cdots,\frac{d(d+1)}{2},~\beta=d+1, \cdots, \frac{d(d+1)}{2}, \alpha\neq\beta;\label{lem5.8-3}
\end{align}
and if $d=2$, then \begin{align}\label{d=2}
|a_{12}|=|a_{21}|&\leq C|\log\varepsilon|;
\end{align}
if $d\geq3$, then \begin{align}\label{dgeq3}
|a_{\alpha \beta}|=|a_{\beta\alpha}|&\leq C,\quad \alpha,\beta=1,\cdots, d,~\alpha\neq\beta.
\end{align}

Consequently, 
\begin{align}\label{det-a4}
\frac{1}{C(\rho_d(\varepsilon))^d}\leq\det A\leq \frac{C}{(\rho_d(\varepsilon))^d},\quad \frac{1}{C}I\leq  D\leq CI.
\end{align}
\end{lemma}

\begin{proof}
{\bf STEP 1.} Proof of (\ref{lem5.8-1}).
In view of  (\ref{ellip}), (\ref{v-alpha1}) and (\ref{v-alpha2}), we have, for $\alpha=1, \cdots, d$,  
\begin{align*}
a_{\alpha\alpha}&=\int_{\Omega}(\mathbb{C}^0e(u_\alpha), e(u_\alpha))dx\leq C\int_{\Omega}|\nabla u_\alpha|^2dx\\
&\leq C\int_{\Omega_R}\frac{dx}{(\varepsilon+|x'|^2)^2}+C\\
&\leq C\int_{|x'|<R}\frac{dx'}{\varepsilon+|x'|^2}+C\\
&=C\int_0^R\frac{\rho^{d-2}}{\varepsilon+\rho^2}d\rho+C
\leq \frac{C}{\rho_d(\varepsilon)},
\end{align*}
and
\begin{align*}
a_{\alpha\alpha}&=\int_{\Omega}(\mathbb{C}^0e(u_\alpha), e(u_\alpha))dx\geq \frac{1}{C}\int_{\Omega}|e(u_\alpha)|^2dx\\
&\geq\frac{1}{C}\int_{\Omega}|e_{\alpha d}(u_\alpha)|^2dx\geq \frac{1}{C}\int_{\Omega_R}|\partial_{x_d}u_\alpha^\alpha|^2dx.
\end{align*}
Notice that $u_\alpha^\alpha|_{\partial{D}_{1}}=\bar{v}|_{\partial{D}_{1}}=1, u_\alpha^\alpha|_{\partial D}=\bar{v}|_{\partial D}=0$, and recalling the definition of $\bar{v}$, $\bar{v}(x',x_d)$ is linear in $x_d$ for fixed $x'$, so $\bar{v}(x', \cdot)$ is harmonic, hence its energy is minimal, that is
\begin{align*}
\int_{h(x')}^{h_1(x')+{\varepsilon}}|\partial_{x_d}u_\alpha^\alpha|^2dx_d\geq
\int_{h(x')}^{h_1(x')+{\varepsilon}}|\partial_{x_d}\bar{v}|^2dx_d=\frac{1}{\varepsilon+h_1(x')-h(x')}.
\end{align*}
Integrating on $B_R(0')$ for $x'$, we obtain
\begin{align*}
\int_{\Omega_R}|\partial_{x_d}u_\alpha^\alpha|^2dx&=\int_{|x'|<R}\int_{h(x')}^{h_1(x')+\varepsilon}|\partial_{x_d}u_\alpha^\alpha|^2dx_ddx'\\
&\geq\frac{1}{C}\int_{|x'|<R}\frac{dx'}{\varepsilon+|x'|^2}\geq\frac{1}{C\rho_d(\varepsilon)}.
\end{align*}
Estimate (\ref{lem5.8-1}) is proved.

{\bf STEP 2.} Proof of (\ref{lem5.8-2}) and (\ref{lem5.8-3}). By means of (\ref{v3}), for $\alpha, \beta=d+1, \cdots, \frac{d(d+1)}{2}$, we have
\begin{align*}
a_{\alpha\beta}&=\int_{\Omega}(\mathbb{C}^0e(u_\alpha), e(u_\beta))dx\leq C\int_{\Omega}|\nabla u_\alpha||\nabla u_\beta|dx\\
&\leq\,C\int_{\Omega_R}\frac{(\varepsilon+|x'|)^2}{(\varepsilon+|x'|^2)^2}dx+C\leq C.
\end{align*}
On the other hand, it follows immediately from Lemma \ref{lemma4.1} that there exists a universal constant $C$ such that
\begin{align*}
a_{\alpha\alpha}\geq\frac{1}{C},\quad\alpha=d+1, \cdots, \frac{d(d+1)}{2}.
\end{align*}

We now consider the elements for $\alpha=1,2,\cdots,d$, $\beta=d+1,d+2,\cdots,\frac{d(d+1)}{2}$. We take the case that $\alpha=1$, $\beta=d+1$ for instance. The other cases are the same. Let $\psi_{d+1}=(x_{2},-x_{1},0,\cdots,0)^{T}$. Then using (\ref{v-alpha1}) and the boundedness of $|\nabla{u}_{\alpha}|$ on $\partial{D}_{1}\setminus B_R$, we have
\begin{align*} 
a_{1{(d+1)}}=&-\int_{\partial{D}_{1}}\frac{\partial u_1}{\partial \nu_0}\large\Big|_{+}\cdot\psi_{d+1}\\
=&-\int_{\partial{D}_{1}\cap B_R}\Big(\lambda(\nabla\cdot u_1)\vec{n}+\mu(\nabla u_\alpha+(\nabla u_1)^T)\vec{n}\Big)\cdot(x_{2},-x_{1},0,\cdots,0)^{T}\\
=&-\int_{\partial{D}_{1}\cap B_R}\left(\lambda\big(\sum_{k=1}^d\partial_{x_k} u_1^k\big)n_1
+\mu\sum_{l=1}^d\left(\partial_{x_1}u_1^l+\partial_{x_l}u_1^1\right)n_l\right)x_{2}\\
&+\int_{\partial{D}_{1}\cap B_R}\left(\lambda\big(\sum_{k=1}^d\partial_{x_k} u_1^k\big)n_2
+\mu\sum_{l=1}^d\left(\partial_{x_2}u_1^l+\partial_{x_l}u_1^2\right)n_l\right)x_{1}
\end{align*}
is bounded for $d\geq2$, so $a_{1(d+1)}$.

Thus, estimates (\ref{lem5.8-2})  and (\ref{lem5.8-3}) are established.

{\bf STEP 3.} Proof of (\ref{d=2}) and (\ref{dgeq3}). Firstly, we estimate $|a_{\alpha\beta}|$  for $\alpha, \beta=1, \cdots, d$ with $\alpha\neq\beta$. By the definition,
\begin{align*}
a_{\alpha\beta}&=a_{\beta\alpha}=-\int_{\partial{D}_{1}}\frac{\partial u_\alpha}{\partial \nu_0}\large\Big|_{+}\cdot\psi_\beta\\
&=-\int_{\partial{D}_{1}}\lambda(\nabla\cdot u_\alpha)n_\beta+\mu\left((\nabla u_\alpha+(\nabla u_\alpha)^T)\vec{n}\right)_\beta\\
&=-\int_{\partial{D}_{1}}\lambda\left(\sum_{k=1}^d\partial_{x_k} u_\alpha^k\right)n_\beta
+\mu\sum_{l=1}^d\left(\partial_{x_\beta}u_\alpha^l+\partial_{x_l}u_\alpha^\beta\right)n_l.
\end{align*}

Denote
$$\mathrm{I}_{\alpha\beta}:=\int_{\partial{D}_{1}\cap B_R}\left(\sum_{k=1}^d\partial_{x_k} u_\alpha^k\right)n_\beta;
$$
and
\begin{align*}
\mathrm{II}_{\alpha\beta}:=&\int_{\partial{D}_{1}\cap B_R} \sum_{l=1}^{d}\left(\partial_{x_\beta}u_\alpha^l+\partial_{x_l}u_\alpha^\beta\right)n_l\\
=&\int_{\partial{D}_{1}\cap B_R} \sum_{l=1}^{d-1}\left(\partial_{x_\beta}u_\alpha^l+\partial_{x_l}u_\alpha^\beta\right)n_l+\int_{\partial{D}_{1}\cap B_R} \partial_{x_\beta}u_\alpha^dn_d+\int_{\partial{D}_{1}\cap B_R} \partial_{x_d}u_\alpha^\beta n_d\\
=&:\,\mathrm{II}_{\alpha\beta}^{1}+\mathrm{II}_{\alpha\beta}^{2}+\mathrm{II}_{\alpha\beta}^{3},
\end{align*}
where $$\vec{n}=\frac{(-\nabla_{x'}h(x'), 1)}{\sqrt{1+|\nabla_{x'}h(x')|^2}}.$$
Due to (\ref{h2}), for $k=1,2,\cdots,d-1$,
\begin{align}\label{n} |n_{k}|=\left|\frac{-\partial_{x_{k}}h(x')}{\sqrt{1+|\nabla_{x'}h(x')|^2}}\right|\leq C|x'|,\quad~\mbox{and}~ |n_d|=\frac{ 1}{\sqrt{1+|\nabla_{x'}h(x')|^2}}\leq1.
\end{align}

For $\alpha=1,2,\cdots,d$, $\beta=1,2,\cdots,d-1$, it follows from (\ref{v-alpha1}) and (\ref{n}) that
\begin{align}\label{a-3}
|\mathrm{I}_{\alpha\beta}|\leq&\int_{\partial{D}_{1}\cap B_R}\left|\left(\sum_{k=1}^d\partial_{x_k} u_\alpha^k\right)n_\beta\right|\nonumber\\
\leq&\int_{\partial{D}_{1}\cap B_R}\frac{C|x'|}{\varepsilon+|x'|^2}
\leq
\begin{cases}
C|\log\varepsilon|,&d=2,\\
C,&d\geq3,
\end{cases}
\end{align}
while, 
\begin{align*} 
|\mathrm{II}_{\alpha\beta}^{1}|\leq\int_{\partial{D}_{1}\cap B_R} \left|\sum_{l=1}^{d-1}\left(\partial_{x_\beta}u_\alpha^l+\partial_{x_l}u_\alpha^\beta\right)n_l\right|
\leq\int_{\partial{D}_{1}\cap B_R}\frac{C|x'|}{\sqrt{\varepsilon+|x'|^2}}
\leq\,C,
\end{align*}
\begin{align*}
|\mathrm{II}_{\alpha\beta}^{2}|\leq\int_{\partial{D}_{1}\cap B_R} \left|\partial_{x_\beta}u_\alpha^dn_d\right|\leq\int_{\partial{D}_{1}\cap B_R}\frac{C}{\sqrt{\varepsilon+|x'|^2}}
\leq\begin{cases}C|\log\varepsilon|,&d=2,\\
C,&d\geq3,
\end{cases}\quad\quad\quad\quad
\end{align*}
and by the definition of $\tilde{u}_\alpha$ and (\ref{cor3.4-1}),
\begin{align*} 
|\mathrm{II}_{\alpha\beta}^{3}|\leq\int_{\partial{D}_{1}\cap B_R}\left|\partial_{x_d} u_\alpha^\beta n_d\right|&\leq\int_{\partial{D}_{1}\cap B_R}|(\partial_{x_d} \tilde{u}_\alpha^\beta)n_d|+\int_{\partial{D}_{1}\cap B_R}|(\partial_{x_d}( u_\alpha^\beta-\tilde{u}_\alpha^\beta)n_d|\nonumber\\
&\leq\int_{\partial{D}_{1}\cap B_R}\frac{C}{\sqrt{\varepsilon+|x'|^2}}
\leq\begin{cases}C|\log\varepsilon|,&d=2,\\
C,&d\geq3.
\end{cases}
\end{align*}
Here we used the fact that $\tilde{u}_\alpha^{\beta}=0$ if $\alpha\neq\beta$. Hence, 
\begin{align*} 
|\mathrm{II}_{\alpha\beta}|
\leq
\begin{cases}C|\log\varepsilon|,&d=2,\\
C,&d\geq3.
\end{cases}
\end{align*}
This, together with (\ref{a-3}), the boundedness of $|\nabla{u}_{\alpha}|$ on $\partial{D}_{1}\setminus B_R$, and the symmetry of $a_{\alpha\beta}=a_{\beta\alpha}$, implies that for $ \alpha,\beta=1,\cdots,d$ with $\alpha\neq\beta$,
\begin{align*}
|a_{\alpha\beta}|=|a_{\beta\alpha}|\leq |\lambda||I_{\alpha\beta}|+|\mu||II_{\alpha\beta}|+C\leq\begin{cases}C|\log\varepsilon|,&d=2,\\
C,&d\geq3.
\end{cases}
\end{align*}
Therefore, (\ref{d=2}) and (\ref{dgeq3}) are proved. (\ref{det-a4}) is an immediate consequence of (\ref{lem5.8-1})--(\ref{dgeq3}). The proof of the Lemma \ref{lemma4.2} is finished.
\end{proof}

\begin{lemma}\label{lemma4.3}
\begin{align}\label{4.18}|b_{\beta}|\leq C,\quad \beta=1, \cdots,\frac{d(d+1)}{2}.
\end{align}
Consequently,
\begin{equation}\label{bound_Pi}
|P^{i}|\leq\,C,\quad\,i=1,2.
\end{equation}
\end{lemma}

\begin{proof}
{\bf STEP 1.} To estimate $|b_\beta|$ for $\beta=1,\cdots,d$. We take $\beta=1$ for instance. The other cases are the same. Denote 
\begin{align*}
 b_1=\sum_{l=1}^{d}\int_{\partial{D}_{1}}\frac{\partial u_{0l}}{\partial \nu_0}\large\Big|_{+}\cdot\psi_1:=\sum_{l=1}^{d}b_{1l}.
\end{align*}
where $u_{0l}$, $l=1,2, \cdots, d$, is defined by \eqref{v01}.
By definition, 
\begin{align*}
 b_{11}&=\int_{\partial{D}_{1}}\frac{\partial u_{01}}{\partial \nu_0}\large\Big|_{+}\cdot\psi_1\\
 &=\int_{\partial{D}_{1}}\left[\lambda(\nabla\cdot u_{01})n_1+\mu\left((\nabla u_{01}+(\nabla u_{01})^T)\vec{n}\right)_1\right]\\
&=\int_{\partial{D}_{1}}\left[\lambda\sum_{k=1}^d\partial_{x_k} u_{01}^kn_1+\mu\sum_{i=1}^d(\partial_{x_1}u_{01}^i
+\partial_{x_i}u_{01}^1)n_i \right].
\end{align*}
Denote
\begin{align*}
\mathrm{I}:=&\int_{\partial{D}_{1}}\sum_{k=1}^d\partial_{x_k} u_{01}^kn_1
=\int_{\partial{D}_{1}}\sum_{k=1}^{d-1}\partial_{x_k} u_{01}^kn_1+\int_{\partial{D}_{1}}\partial_{x_d} u_{01}^dn_1=:\mathrm{I}_{1}+\mathrm{I}_{2},
\end{align*}
and
\begin{align*}
\mathrm{II}:=&\int_{\partial{D}_{1}}\sum_{i=1}^d(\partial_{x_1}u_{01}^i
+\partial_{x_i}u_{01}^1)n_i\\
=&\int_{\partial{D}_{1}}\sum_{i=1}^{d-1}(\partial_{x_1}u_{01}^i
+\partial_{x_i}u_{01}^1)n_i+\int_{\partial{D}_{1}}\partial_{x_1}u_{01}^dn_d
+\int_{\partial{D}_{1}}\partial_{x_d}u_{01}^1n_d\\
=&:\mathrm{II}_{1}+\mathrm{II}_{2}+\mathrm{II}_{3}.
\end{align*}

According to (\ref{step3-4})-(\ref{step3-5}), 
 \begin{align*}
&|\mathrm{I}_{1}|\leq\left|\int_{\partial{D}_{1}}\sum_{k=1}^{d-1}\partial_{x_k} u_{01}^kn_1\right|\leq \int_{\partial{D}_{1}\cap B_R}C|x'|+C\leq C;\\
&|\mathrm{I}_{2}|\leq\left|\int_{\partial{D}_{1}} \partial_{x_d} u_{01}^dn_1\right|\leq\int_{\partial{D}_{1}\cap B_R}\frac{C|\nabla_{x'}\varphi^1(0)||x'|^2}{\varepsilon+|x'|^2}
+C\leq C.
 \end{align*}
 So that, 
\begin{align}\label{b1}
|\mathrm{I}| &\leq |\mathrm{I}_{1}|+ |\mathrm{I}_{2}|\leq C.
 \end{align}

By (\ref{step3-4}), (\ref{n}) and the definition of $\tilde{u}_{01}$,  
 \begin{align}\label{b2}
|\mathrm{II}_{1}|\leq&\left|\int_{\partial{D}_{1}}\sum_{i=1}^{d-1}(\partial_{x_1}u_{01}^i
+\partial_{x_i}u_{01}^1)n_i \right|
\leq\int_{\partial{D}_{1}\cap B_R}C|x'|+C\leq C,
 \end{align}
 and
 \begin{align}
 |\mathrm{II}_{2}|\leq&\int_{\partial{D}_{1}}|\partial_{x_1} u_{01}^dn_d|\nonumber\\
 \leq&\int_{\partial{D}_{1}\cap B_R}|\partial_{x_1} \tilde{u}_{01}^dn_d|+\int_{\partial{D}_{1}\cap B_R}|\partial_{x_1} (u_{01}^d-\tilde{u}_{01}^d)n_d|+C\nonumber\\
\leq&C|\partial{D}_{1}\cap B_R|+C
\leq C.\label{b3}
\end{align}

Now, we need only to estimate $\mathrm{II}_{3}$. Note that
\begin{align*}
\mathrm{II}_{3}&=\int_{\partial{D}_{1}\cap B_R}\partial_{x_{d}} \tilde{u}_{01}^1n_d+\int_{\partial{D}_{1}\cap B_R}\partial_{x_{d}}( u_{01}^1-\tilde{u}_{01}^1) n_d=:\mathrm{II}_{3}^{1}+\mathrm{II}_{3}^{2}.
\end{align*}
By the definitions of $\tilde{u}_{01}^1$ and $\bar{v}$,
\begin{align*}
\partial_{x_{d}}\tilde{u}_{01}^1=-(\varphi^1(x',h(x'))-\varphi^1(0))\partial_{x_{d}}\bar{v}=-\frac{\varphi^1(x',h(x'))-\varphi^1(0)}{\varepsilon+h_{1}(x')-h(x')}.
\end{align*}
From the  expression of
$\partial{D}_{1}\cap B_R:  x_{d}=\varepsilon+h_1(x'),\ |x'|<R$, we have $dS=\sqrt{1+|\nabla_{x'}h_1(x')|^2}dx'$. Then, by the Taylor expansion (\ref{taylor}), we have
 \begin{align*}
\mathrm{II}_{3}^{1}=&\int_{\partial{D}_{1}\cap B_R}\partial_{x_{d}} \tilde{u}_{01}^1n_d\\
 &=\int_{|x'|<R}\frac{-(\varphi^1(x',h(x'))-\varphi^1(0))}{\varepsilon+h_1(x')-h(x')}dx'\\
 &=-\int_{|x'|<R}\frac{\nabla_{x'}\varphi^1(0)x'+O(|x'|^2)}{\varepsilon+h_1(x')-h(x')}dx'.
\end{align*}
Since
\begin{align}\label{bound}\frac{1}{C(\varepsilon+|x'|^2)}\leq\frac{1}{\varepsilon+h_1(x')-h(x')}\leq\frac{C}{\varepsilon+|x'|^2},\quad |x'|\leq R,\end{align}
it follows that
\begin{align*}
\left|\int_{|x'|<R}\frac{O(|x'|^2)}{\varepsilon+h_1(x')-h(x')}dx'\right|\leq C.
\end{align*}
While, according to (\ref{bound}), we have
\begin{align*}
&\int_{|x'|<R}\frac{\nabla_{x'}\varphi^1(0)x'}{\varepsilon+h_1(x')-h(x')}dx'\\
&=\int_{|x'|<R}\frac{\nabla_{x'}\varphi^1(0)x'}{\varepsilon+\frac{1}{2}{x'}^T(\nabla^2_{x'}(h_1-h)(0'))x'}dx'\\
&+\int_{|x'|<R}\frac{O(|x'|^{3+\gamma})}{(\varepsilon+\frac{1}{2}{x'}^T(\nabla^2_{x'}(h_1-h)(0'))x')(\varepsilon+h_1(x')-h(x'))}dx'.
\end{align*}
For the positive matrix $(\nabla^2_{x'}(h_1-h)(0'))$, there exists orthogonal matrix $O$, such that
\begin{align*}
O^T(\nabla^2_{x'}(h_1-h)(0'))O=diag (\lambda_1,\cdots,\lambda_{d-1}),
\end{align*}
where $\lambda_i\geq \kappa_1,\ i=1,\cdots,d-1.$ Under the orthogonal transform $y'=Ox'$, we obtain
\begin{align*}
&\int_{|x'|<R}\frac{\nabla_{x'}\varphi^1(0)x'}{\varepsilon+\frac{1}{2}{x'}^T(\nabla^2_{x'}(h_1-h)(0'))x'}dx'=\int_{|y'|<R}\frac{\nabla_{x'}\varphi^1(0)O^Ty'}{\varepsilon+\Sigma_{i=1}^{d-1}\lambda_iy_i^2}dy'=0,
\end{align*}
and
\begin{align*}
\left|\int_{|x'|<R}\frac{O(|x'|^{3+\gamma})}{(\varepsilon+\frac{1}{2}{x'}^T(\nabla^2_{x'}(h_1-h)(0'))x')(\varepsilon+h_1(x')-h(x'))}dx'\right|\leq C.
\end{align*}
Therefore,
\begin{align}\label{b4}
 &|\mathrm{II}_{3}^{1}|=\left|\int_{\partial{D}_{1}\cap B_R}\partial_{x_{d}} \tilde{u}_{01}^1n_d\right|\leq C.
\end{align}
On the other hand, in view of \eqref{step3-3},
\begin{align*}
|\mathrm{II}_{3}^{2}|\leq\left|\int_{\partial{D}_{1}\cap B_R}\partial_{x_{d}}( u_{01}^1-{\tilde{u}}_{01}^1) n_d\right|\leq C.
 \end{align*}
This, together with (\ref{b4}), implies that
\begin{align}\label{b5}
|\mathrm{II}_{3}|\leq C.
 \end{align}
Combining (\ref{b1})-(\ref{b3}) and (\ref{b5}), we have
\begin{align*}
|b_{11}|\leq  C.
\end{align*}

Next, for $l=2,\cdots, d,$
\begin{align*}
 b_{1l}&=\int_{\partial{D}_{1}}\frac{\partial u_{0l}}{\partial \nu_0}\large\Big|_{+}\cdot\psi_1\\
 &=\int_{\partial{D}_{1}}\left[\lambda(\nabla\cdot u_{0l})n_1+\mu\left((\nabla u_{0l}+(\nabla u_{0l})^T)\vec{n}\right)_1\right]\\
&=\int_{\partial{D}_{1}}\left[\lambda\sum_{k=1}^d\partial_{x_k} u_{0l}^kn_1+\mu\sum_{i=1}^d(\partial_{x_1}u_{0l}^i
+\partial_{x_i}u_{0l}^1)n_i\right],
\end{align*}
Similarly, making use of (\ref{step3-4}), (\ref{step3-5}) and (\ref{n}), we have
\begin{align}\label{b1'}
 \left|\int_{\partial{D}_{1}}\sum_{k=1}^d\partial_{x_k} u_{0l}^kn_1\right|\ &\leq  \left|\int_{\partial{D}_{1}}\sum_{k=1}^{d-1}\partial_{x_k} u_{0l}^kn_1\right|+ \left|\int_{\partial{D}_{1}} \partial_{x_d} u_{0l}^dn_1\right|\leq C.
 \end{align}
and recalling the definition of $\tilde{u}_{0l}$, and $\tilde{u}_{0l}^{1}=0$,
\begin{align*}
&\left|\int_{\partial{D}_{1}}\sum_{i=1}^d(\partial_{x_1}u_{0l}^i
+\partial_{x_i}u_{0l}^1)n_i \right|\\
\leq& \left|\int_{\partial{D}_{1}}\sum_{i=1}^{d-1}(\partial_{x_1}u_{0l}^i
+\partial_{x_i}u_{0l}^1)n_i \right|+\left|\int_{\partial{D}_{1}}\partial_{x_1} (u_{0l}^d-\tilde{u}_{0l}^d)n_d\right|\\
&+\left|\int_{\partial{D}_{1}}\partial_{x_1} \tilde{u}_{0l}^dn_d\right|
 +\left|\int_{\partial{D}_{1}}\partial_{x_{d}} (u_{0l}^1-\tilde{u}_{0l}^1)n_d\right|\\
\leq& \,C.
 \end{align*}
This implies that
$$|b_{1l}|\leq  C,\quad l=2,\cdots,d.$$
Hence,
$$|b_{1}|\leq  C.$$

{\bf STEP 2.} To estimate $|b_\beta|$ for $\beta=d+1,\cdots,\frac{d(d+1)}{2}$. By using (\ref{v3}) and (\ref{v0}), we have
\begin{align*}
|b_\beta|&=\left|\int_{\Omega}\left(\mathbb{C}^0e(u_0), e(u_{\beta})\right)dx\right|\\
&\leq C\int_{\Omega}|\nabla u_0||\nabla u_{\beta}|dx\\
&\leq\int_{\Omega_{R}}\frac{C|\nabla_{x'}\varphi(0)||x'|(\varepsilon+|x'|)}{(\varepsilon+|x'|^2)^2}dx+C\\
&\leq C.
\end{align*}
The proof of Lemma \ref{lemma4.3} is completed.
\end{proof}

\begin{proof}[Proof of Proposition \ref{prop1}]
{\bf Step 1.} Proof of (\ref{C3}). 

Let $u^{\varepsilon}$ be the solution of (\ref{main}). By theorem 6.6 in the appendix in \cite{BLL}, $u^{\varepsilon}$ is the minimizer of
\begin{align*}
I_\infty[u]:=\frac{1}{2}\int_{\Omega}\left(\mathbb{C}^0e(u), e(u)\right)dx
\end{align*}
on $\mathcal{A}$ defined by \eqref{def_A}. It follows that
$$\|u^{\varepsilon}\|_{H^1(\Omega)}^2\leq C\|e(u^{\varepsilon})\|_{L^2(\Omega)}^2\leq CI_\infty[u^{\varepsilon}]\leq C.$$
By the Sobolev trace embedding theorem,
$$\|u^{\varepsilon}\|_{L^2(\partial{D}_{1}\cap B_R)}\leq C.$$
Recalling that
\begin{align*}
u^{\varepsilon}=\sum_{\alpha=1}^{\frac{d(d+1)}{2}}C^\alpha \psi_\alpha,\quad \hbox{on}\ \partial{D}_{1}.
\end{align*}
If $\mathcal{C}:=(C^1, C^2,\cdots, C^{\frac{d(d+1)}{2}})=0$, there is nothing to prove. Otherwise,
\begin{align}\label{CC}
C\geq|\mathcal{C}|\left\|\sum_{\alpha=1}^{\frac{d(d+1)}{2}}\widehat{C}_\alpha \psi_\alpha\right\|_{L^2(\partial{D}_{1}\cap B_R)},
\end{align}
where $\widehat{C}_\alpha=\frac{C^\alpha}{|\mathcal{C}|}$ and $|\widehat{C}|=1$. It is easy to see that
\begin{align}\label{lowerbound}
\left\|\sum_{\alpha=1}^{\frac{d(d+1)}{2}}\widehat{C}_\alpha \psi_\alpha\right\|_{L^2(\partial{D}_{1}\cap B_R)}\geq\frac{1}{C}.
\end{align}
Indeed, if not, along a subsequence $\varepsilon\rightarrow0,~\widehat{C}_\alpha\rightarrow\overline{C}_\alpha$, and
\begin{align*}
\left\|\sum_{\alpha=1}^{\frac{d(d+1)}{2}}\overline{C}_\alpha \psi_\alpha\right\|_{L^2(\partial{D}_{1}^*\cap B_R)}=0,
\end{align*}
where $\partial{D}_{1}^*$ is the limit of $\partial{D}_{1}$ as $\varepsilon\rightarrow0$ and $|\overline{C}|=1.$ This implies
$$\sum_{\alpha=1}^{\frac{d(d+1)}{2}}\overline{C}_\alpha \psi_\alpha=0\quad\hbox{on}\ \partial{D}_{1}^*\cap B_R.$$
But $\{\psi_\alpha|_{\partial{D}_{1}^*\cap B_R}\} $ is easily seen to be linear independent, according to lemma 6.1 in the appendix of \cite{BLL2}, we must have $\overline{C}=0$. This is a contradiction.
(\ref{C3}) follows from (\ref{CC}) and (\ref{lowerbound}).

{\bf Step 2.} Proof of (\ref{C4}). According to Lemma \ref{lemma4.1},  the matrix 
\begin{gather*}  \begin{pmatrix}A&B\\B^T&D\end{pmatrix}\end{gather*}  is positive definite, so invertable. Moreover, 
\begin{align}
A\geq\frac{1}{C}I_{d\times d},\qquad
D\geq\frac{1}{C}I_{\frac{d(d-1)}{2}\times \frac{d(d-1)}{2}}.\label{D}
\end{align}
Therefore, from \eqref{eq5.3}, we have
\begin{gather*}  
\begin{pmatrix} X^1\\  X^2 \end{pmatrix}=\begin{pmatrix}A&B\\B^T&D\end{pmatrix}^{-1}\begin{pmatrix}P^1\\P^2\end{pmatrix}.
\end{gather*}

For $d\geq4$, it is easy to see from Lemma \ref{lemma4.1} and Lemma \ref{lemma4.3} that
$$|X^{1}|\leq\,C.$$

Next, we prove \eqref{C4} for $d=2,3$. By lemma 6.2 in Appendix of \cite{BLL2} and Lemma \ref{lemma4.2},
\begin{gather*}\begin{pmatrix}A&B \\  B^T&D\end{pmatrix}^{-1}=\begin{pmatrix}A^{-1}& 0\\  0&D^{-1} \end{pmatrix}+\begin{pmatrix}Errors\end{pmatrix},
\end{gather*}
where
\begin{align}\label{errors}
\left|\begin{pmatrix}Errors\end{pmatrix}\right|\sim o(\rho_d(\varepsilon)).
\end{align}
Then,  
\begin{gather*}  \begin{pmatrix} X^1\\  X^2 \end{pmatrix}=\begin{pmatrix}A^{-1}& 0\\  0&D^{-1} \end{pmatrix}\begin{pmatrix}P^1\\P^2\end{pmatrix}+\begin{pmatrix}Errors\end{pmatrix}\begin{pmatrix}P^1\\P^2\end{pmatrix}
=\begin{pmatrix}A^{-1}P^1\\D^{-1}P^2\end{pmatrix}+\begin{pmatrix}Errors\end{pmatrix}.
\end{gather*}
Therefore,
\begin{gather}\label{X1}   X^1=A^{-1}P^1+Errors=\frac{1}{\det A}A^*P^1+Errors,
\end{gather}
where $A^*=\left(a_{\alpha\beta}^*\right)$ is the adjoint matrix of $A$. Following Lemma \ref{lemma4.2}, it is clear that
\begin{gather*}A^*\sim\begin{pmatrix}\frac{\tilde{c}_{1}}{(\rho_d(\varepsilon))^{d-1}}&\cdots&o\left(\frac{1}{(\rho_d(\varepsilon))^{d-1}}\right) \\  \vdots & \ddots& \vdots\\ o\left(\frac{1}{(\rho_d(\varepsilon))^{d-1}}\right)& \cdots&
\frac{\tilde{c}_{d}}{(\rho_d(\varepsilon))^{d-1}}\end{pmatrix},
\end{gather*} for some constants $\tilde{c}_\alpha\neq0, \alpha=1,\cdots,d$, independent of $\varepsilon$.
In view of \eqref{det-a4} and \eqref{bound_Pi}, we  obtain
\begin{align*}
 |X^1| \leq C\rho_d(\varepsilon).
\end{align*}
Therefore,
\begin{align*}
 |C^{\alpha}-\varphi^{\alpha}(0)| \leq C\rho_d(\varepsilon), \quad \alpha=1,\cdots, d.
\end{align*}
Proposition \ref{prop1} is established.
\end{proof}

\vspace{1cm}

\section{Proof of Theorem \ref{thm1.2} (Lower bound)}\label{sec_5}

In order to prove  Theorem \ref{thm1.2}, we first prove  $b_{\beta}\to\,b_{\beta}^*$, as $\varepsilon\to0$, $\beta=1,\cdots, d$.

\begin{lemma}\label{lemma5.1}
For $d\geq3$, $\beta=1,2,\cdots,d$, 
\begin{align*}
|b_\beta-b_\beta^*|\leq C\left(|\nabla_{x'}\varphi(0)|+\|\nabla^2\varphi\|_{L^\infty(\partial{D})}\right)\varepsilon^{\gamma_{d}};
\end{align*}
for $d=2$, if $\nabla_{x'}\varphi^\beta(0)=0$ for $\beta=1$ or $2$, then, for $\alpha\neq\beta$, 
\begin{align*}
|b_\beta-b_\beta^*|\leq C\left(|\nabla_{x'}\varphi^{\alpha}(0)|+\|\nabla^2\varphi\|_{L^\infty(\partial{D})}\right)\varepsilon^{\gamma_{2}},
\end{align*}
where
$$\gamma_{d}=\begin{cases}
\frac{d-2}{2(d-1)},&d\geq3,\\
\frac{1}{6},&d=2.
\end{cases}
$$
Consequently, 
\begin{align*}
b_{\beta}\rightarrow b_{\beta}^*,\quad\hbox{as}\ \varepsilon\rightarrow 0,\quad\beta=1,2,\cdots,d.
\end{align*}
\end{lemma}

\begin{proof}We here prove the case $\beta=1$ for instance. The  other cases are the same. It follows from the definitions of $u_{0}$ and $u_{1}$ and the integration by parts formula (\ref{eu}) that
\begin{align*}
b_1&=\int_{\partial{D}_{1}}\frac{\partial u_0}{\partial \nu_0}\large|_{+}\cdot u_1=\int_{\Omega}\left(\mathbb{C}^0e(u_1), e(u_0)\right)=\int_{\partial{D}}\frac{\partial u_1}{\partial \nu_0}\large|_{+}\cdot (\varphi(x)-\varphi(0)).\\
\end{align*}
Similarly,
\begin{align*}
b_1^*&=\int_{\partial{D}_{1}^*}\frac{\partial u^*_0}{\partial \nu_0}\large|_{+}\cdot\psi_1=\int_{\partial{D}}\frac{\partial u_1^*}{\partial \nu_0}\large|_{+}\cdot (\varphi(x)-\varphi(0)),\\
\end{align*}
where $u_1^*$ satisfies
\begin{align}\label{u1*}
\begin{cases}
  \mathcal{L}_{\lambda,\mu}u_1^*=0,\quad&
\hbox{in}\  \Omega^*,  \\
u_1^*=\psi_1,\ &\hbox{on}\ \partial{D}_{1}^*\setminus\{0\},\\
u_1^*=0,&\hbox{on} \ \partial{D}.
\end{cases}
\end{align}
Thus,
$$b_{1}-b_{1}^{*}=\int_{\partial{D}}\frac{\partial (u_1-u_{1}^{*})}{\partial \nu_0}\Big|_{+}\cdot (\varphi(x)-\varphi(0)).$$

Similarly as before, in order to estimate the difference $u_1-u_{1}^{*}$, we introduce two auxiliary functions
\begin{gather*} \tilde{u}_1=\begin{pmatrix}
\bar{v}\\0\\
\vdots\\
0\end{pmatrix},\qquad\mbox{and}\quad~~ \tilde{u}_1^*=\begin{pmatrix}
\bar{v}^*\\
0\\
\vdots\\
0\end{pmatrix},
\end{gather*}
where $\bar{v}$ is defined in Section \ref{sec_3}, and $\bar{v}^*$ satisfies $\bar{v}^*=1$ on $\partial{D}_{1}^*\setminus\{0\}$, $\bar{v}^*=0$ on $\partial D$, and
\begin{align*}
\bar{v}^*=\frac{x_{d}-h(x')}{h_1(x')-h(x')},\quad\hbox{on}\ \Omega_R^*,\qquad\|\bar{v}^*\|_{C^{2}(\overline{\Omega^*\setminus\Omega_{R/2}^*})}\leq\,C,
\end{align*}
where $\Omega_r^*:=\left\{~x\in\Omega^*~\big|~ |x'|<r~\right\}$, for $r<R$.
By (\ref{h2}) and (\ref{h1}),   we have, for $x\in\Omega_R^*$,
\begin{align}\label{estimate1}
&|\nabla_{x'}(\tilde{u}_1^1-\tilde{u}_1^{*1})|\leq\frac{C}{|x'|},
\end{align}
and
\begin{align}\label{estimate1a}
|\partial_{x_{d}}(\tilde{u}_1^1-\tilde{u}_1^{*1})|=\Big|\frac{1}{h_1(x')-h(x')}-\frac{1}{\varepsilon+h_1(x')-h(x')}\Big|\leq\frac{C\varepsilon}{|x'|^2(\varepsilon+|x'|^{2})}.
\end{align}
Applying Corollary \ref{corol3.4} to (\ref{u1*}), we obtain
\begin{align}\label{u1*-tildeu1*}
|\nabla (u_1^*-\tilde{u}_1^*)(x)|\leq \frac{C}{|x'|},\quad   x\in\Omega_R^*;
\end{align}
and
\begin{align}\label{u-1*}
|\nabla_{x'}u_1^*(x)|\leq \frac{C}{|x'|},\quad |\partial_{x_{d}}u_1^*(x)|\leq \frac{C}{|x'|^2},\quad x\in\Omega_R^*.
\end{align}
Define a cylinder
$$\mathcal{C}_{r}:=\left\{x\in\mathbb{R}^{d}~\big|~|x'|<r,0\leq\,x_{n}\leq\varepsilon+2\max_{|x'|=r}h_{1}(x')\right\},$$
for $r<R_{0}$. Next, we divide into two steps to estimate the difference $u_1-u_{1}^{*}$.

{\bf STEP 1.} Notice that $u_1-u_1^*$ satisfies
\begin{align*}
\begin{cases}
  \mathcal{L}_{\lambda,\mu}(u_1-u_1^*)=0,\quad&
\hbox{in}\   D\setminus(\overline{D_{1}\cup D_{1}^{*}}),  \\
u_1-u_1^*=\psi_1-u_1^*,\ &\hbox{on}\ \partial{D}_{1} \setminus D_{1}^{*},\\
u_1-u_1^*=u_1-\psi_{1},&\hbox{on} \ \partial{D}_{1}^{*}\setminus (D_{1}\cup\{0\}),\\
u_1-u_1^*=0,&\hbox{on} \ \partial D.\\
\end{cases}
\end{align*}
We first estimate $|u_1-u_1^*|$ on $\partial(D_{1}\cup{D}_{1}^{*})\setminus\mathcal{C}_{\varepsilon^{\gamma}}$, where $0<\gamma<1/2$ to be determined later. For $\varepsilon$ sufficiently small, in view of the definition of $u_1^*$,
\begin{align*}
|\partial_{x_{d}}u_1^*(x)|\leq C,\quad  x\in\Omega^*\setminus\Omega_R^*,
\end{align*}
we have, for $x\in \partial{D}_{1} \setminus D_{1}^{*}$, 
\begin{align}\label{boundary1}|(u_1-u_1^*)(x', x_{d})|&=|u_1^*(x',x_{d}-\varepsilon)-u_1^*(x', x_{d})|\leq C\varepsilon.
\end{align}
For $x\in\partial{D}_{1}^*\setminus(D_{1}\cup \mathcal{C}_{\varepsilon^{\gamma}})$, by (\ref{v12}),
\begin{align}\label{boundary2}|(u_1-u_1^*)(x', x_{d})|&=|u_1(x',x_{d})-u_1(x', x_{d}+\varepsilon)|\nonumber\\
&\leq \frac{C\varepsilon}{\varepsilon+|x'|^2}\leq C\varepsilon^{1-2\gamma}.
\end{align}
By using (\ref{estimate1a}), (\ref{cor3.4-1}) and (\ref{u1*-tildeu1*}), we have, for $x\in\Omega_{R}^{*}$ with $|x'|=\varepsilon^{\gamma}$,
\begin{align*}
\left|\partial_{x_{d}}(u_1-u_1^*)(x',x_{d})\right|
=&\,\left|\partial_{x_{d}}(\tilde{u}_1-\tilde{u}_1^*)+\partial_{x_{d}}(u_1-\tilde{u}_1)
+\partial_{x_{d}}(u_1^*-\tilde{u}_1^*)\right|(x',x_{d})\nonumber\\
\leq&\,\frac{C\varepsilon}{|x'|^{2}(\varepsilon+|x'|^{2})}+\frac{C}{|x'|}\nonumber\\
\leq&\,\frac{C}{\varepsilon^{4\gamma-1}}+\frac{C}{\varepsilon^{\gamma}}.
\end{align*}
Thus, for $x\in\Omega_{R}^{*}$ with $|x'|=\varepsilon^{\gamma}$, recalling $u_1-u_1^*=0$ on $\partial D$, we have
\begin{align}\label{boundary3}
|(u_1-u_1^*)(x',x_{d})|=&|(u_1-u_1^*)(x',x_{d})-(u_1-u_1^*)(x',h(x'))|\nonumber\\
\leq&\sup_{h(x')\leq\,x_{d}\leq\,h_{1}(x')}\big|\partial_{x_{d}}(u_1-u_1^*)(x',x_{d})\big|_{|x'|=\varepsilon^{\gamma}}\cdot (h_1(x')-h(x'))\nonumber\\
\leq&(\frac{C}{\varepsilon^{4\gamma-1}}+\frac{C}{\varepsilon^{\gamma}})\cdot \varepsilon^{2\gamma}\leq C(\varepsilon^{1-2\gamma}+\varepsilon^{\gamma}).
\end{align}
Letting $1-2\gamma=\gamma$, we take $\gamma=1/3$. Combining (\ref{boundary1}), (\ref{boundary2}) and (\ref{boundary3}), and recalling $u_1-u_1^*=0$ on $\partial D$, we obtain
\begin{align*}
|(u_1-u_1^*)(x)|\leq C\varepsilon^{1/3},\quad x\in\partial(D\setminus(\overline{D_{1}\cup D_{1}^{*}\cup \mathcal{C}_{\sqrt[3]{\varepsilon}}})).
\end{align*}
Applying the maximum principle for Lam\'{e} systems, see \cite{MMN},
\begin{align*}
|(u_1-u_1^*)(x)|\leq C\varepsilon^{1/3},\quad\mbox{in}~~ D\setminus\overline{(D_{1}\cup D_{1}^{*}\cup \mathcal{C}_{\sqrt[3]{\varepsilon}})}.
\end{align*}
Then using the standard interior and boundary estimates for Lam\'e system, we have, for any $0<\tilde{\gamma}<1/3$,\begin{align*}
|\nabla(u_1-u_1^*)(x)|\leq C\varepsilon^{\tilde{\gamma}},\quad\mbox{in}~~ D\setminus\overline{(D_{1}\cup D_{1}^{*}\cup \mathcal{C}_{\varepsilon^{\frac{1}{3}-\tilde{\gamma}}})}.
\end{align*}
This implies that 
\begin{align}\label{B_out}
|\mathcal{B}^{out}|:=\Big|\int_{\partial{D}\setminus \mathcal{C}_{\varepsilon^{\frac{1}{3}-\tilde{\gamma}}}}\frac{\partial(u_1-u_1^{*})}{\partial\nu_0}\large\Big|_{+}\cdot (\varphi(x)-\varphi(0))\Big|\leq C\varepsilon^{\tilde{\gamma}},
\end{align}
where $0<\tilde{\gamma}<1/3$ will be determined later.

{\bf STEP 2.} In the following, we estimate 
\begin{align*}
\mathcal{B}^{in}:=&\int_{\partial{D}\cap \mathcal{C}_{\varepsilon^{\frac{1}{3}-\tilde{\gamma}}}}\frac{\partial(u_1-u_1^{*})}{\partial\nu_0}\large\Big|_{+}\cdot (\varphi(x)-\varphi(0))\\
=&\int_{\partial{D}\cap \mathcal{C}_{\varepsilon^{\frac{1}{3}-\tilde{\gamma}}}}\frac{\partial(\tilde{u}_1-\tilde{u}_1^{*})}{\partial\nu_0}\large\Big|_{+}\cdot (\varphi(x)-\varphi(0))\\
&+\int_{\partial{D}\cap \mathcal{C}_{\varepsilon^{\frac{1}{3}-\tilde{\gamma}}}}\frac{\partial(w_1-w_1^{*})}{\partial\nu_0}\large\Big|_{+}\cdot (\varphi(x)-\varphi(0))\\
=&:\mathcal{B}_{\tilde{u}}+\mathcal{B}_{w}
\end{align*}
where $w_1=u_1-\tilde{u}_1,\ w_1^*=u_1^*-\tilde{u}_1^*$. By definition, 
\begin{align*}
\mathcal{B}_{\tilde{u}}=&\int_{\partial{D}\cap \mathcal{C}_{\varepsilon^{\frac{1}{3}-\tilde{\gamma}}}}\left\{\lambda \sum_{k=1}^d\partial_{x_1}(\tilde{u}_1^1-\tilde{u}_1^{*1})n_k(\varphi^k(x)-\varphi^k(0))\right.\\
&\quad\left.+\mu\sum_{k=1}^d\partial_{x_k}(\tilde{u}_1^1-\tilde{u}_1^{*1})\Big[n_1(\varphi^k(x)-\varphi^k(0))+n_k(\varphi^1(x)-\varphi^1(0))
\Big]\right\}\\
=&:\lambda\left(\mathcal{B}_{\tilde{u}}^{1}+\mathcal{B}_{\tilde{u}}^{2}\right)+\mu\left(\mathcal{B}_{\tilde{u}}^{3}+\mathcal{B}_{\tilde{u}}^{4}+\mathcal{B}_{\tilde{u}}^{5}+\mathcal{B}_{\tilde{u}}^{6}\right),
\end{align*}
where
$$\mathcal{B}_{\tilde{u}}^{1}:=\int_{\partial{D}\cap \mathcal{C}_{\varepsilon^{\frac{1}{3}-\tilde{\gamma}}}}\sum_{k=1}^{d-1}\partial_{x_1}(\tilde{u}_1^1-\tilde{u}_1^{*1})n_k(\varphi^k(x)-\varphi^k(0)),$$
$$\mathcal{B}_{\tilde{u}}^{2}:=\int_{\partial{D}\cap \mathcal{C}_{\varepsilon^{\frac{1}{3}-\tilde{\gamma}}}}\partial_{x_1}(\tilde{u}_1^1-\tilde{u}_1^{*1})n_d(\varphi^d(x)-\varphi^d(0)),$$
$$\mathcal{B}_{\tilde{u}}^{3}:=\int_{\partial{D}\cap \mathcal{C}_{\varepsilon^{\frac{1}{3}-\tilde{\gamma}}}}\sum_{k=1}^{d-1}\partial_{x_k}(\tilde{u}_1^1-\tilde{u}_1^{*1})n_1(\varphi^k(x)-\varphi^k(0)),$$
$$\mathcal{B}_{\tilde{u}}^{4}:=\int_{\partial{D}\cap \mathcal{C}_{\varepsilon^{\frac{1}{3}-\tilde{\gamma}}}}\partial_{x_{d}}(\tilde{u}_1^1-\tilde{u}_1^{*1})n_1(\varphi^d(x)-\varphi^d(0)),$$
$$\mathcal{B}_{\tilde{u}}^{5}:=\int_{\partial{D}\cap \mathcal{C}_{\varepsilon^{\frac{1}{3}-\tilde{\gamma}}}}\sum_{k=1}^{d-1}\partial_{x_k}(\tilde{u}_1^1-\tilde{u}_1^{*1})n_k(\varphi^1(x)-\varphi^1(0)),$$
and
$$\mathcal{B}_{\tilde{u}}^{6}:=\int_{\partial{D}\cap \mathcal{C}_{\varepsilon^{\frac{1}{3}-\tilde{\gamma}}}}\partial_{x_{d}}(\tilde{u}_1^1-\tilde{u}_1^{*1})n_d(\varphi^1(x)-\varphi^1(0)).$$

According to (\ref{estimate1}), (\ref{estimate1a}) and the Taylor expansion of $\varphi^k(x)$,
\begin{align}\label{est1}
&\left|\mathcal{B}_{\tilde{u}}^{1}\right|+\left|\mathcal{B}_{\tilde{u}}^{3}\right|+\left|\mathcal{B}_{\tilde{u}}^{5}\right|\nonumber\\
\leq&\int_{\partial{D}\cap \mathcal{C}_{\varepsilon^{\frac{1}{3}-\tilde{\gamma}}}}\frac{C}{|x'|}\cdot|x'|\cdot|\nabla_{x'}\varphi(0)||x'|+\int_{\partial{D}\cap \mathcal{C}_{\varepsilon^{\frac{1}{3}-\tilde{\gamma}}}}C\|\nabla^2\varphi\|_{L^\infty(\partial{D})}|x'|^2\nonumber\\
\leq&\,C|\nabla_{x'}\varphi (0)|\varepsilon^{(\frac{1}{3}-\tilde{\gamma})d}+C\varepsilon^{(\frac{1}{3}-\tilde{\gamma})(d+1)}\|\nabla^2\varphi\|_{L^\infty(\partial{D})};
\end{align}
\begin{align}\label{est2}
\left|\mathcal{B}_{\tilde{u}}^{2}\right|\leq&\int_{\partial{D}\cap \mathcal{C}_{\varepsilon^{\frac{1}{3}-\tilde{\gamma}}}}\frac{C}{|x'|}\cdot|\nabla_{x'}\varphi(0)||x'|+\int_{\partial{D}\cap \mathcal{C}_{\varepsilon^{\frac{1}{3}-\tilde{\gamma}}}}C\|\nabla^2\varphi\|_{L^\infty(\partial{D})}|x'|\nonumber\\
\leq&\,C|\nabla_{x'}\varphi(0)|\varepsilon^{(\frac{1}{3}-\tilde{\gamma})(d-1)}+C\|\nabla^2\varphi\|_{L^\infty(\partial{D})}\varepsilon^{(\frac{1}{3}-\tilde{\gamma})d};
\end{align}
and
\begin{align}\label{est3}
\left|\mathcal{B}_{\tilde{u}}^{4}\right|
\leq&\int_{\partial{D}\cap \mathcal{C}_{\varepsilon^{\frac{1}{3}-\tilde{\gamma}}}}\frac{C}{|x'|^2}\cdot|x'|\cdot|\nabla_{x'}\varphi(0)||x'|+\int_{\partial{D}\cap \mathcal{C}_{\varepsilon^{\frac{1}{3}-\tilde{\gamma}}}}C\|\nabla^2\varphi\|_{L^\infty(\partial{D})}|x'|\nonumber\\
\leq&\,C|\nabla_{x'}\varphi(0)|\varepsilon^{(\frac{1}{3}-\tilde{\gamma})(d-1)}+C\|\nabla^2\varphi\|_{L^\infty(\partial{D})}\varepsilon^{(\frac{1}{3}-\tilde{\gamma})d}.
\end{align}
For $d=2$ if $\nabla_{x'}\varphi^1(0)=0$, we have
\begin{align}\label{est4}
\left|\mathcal{B}_{\tilde{u}}^{6}\right|
\leq& \int_{\partial{D}\cap \mathcal{C}_{\varepsilon^{\frac{1}{3}-\tilde{\gamma}}}}C\|\nabla^2\varphi^1\|_{L^\infty(\partial{D})}\leq C\|\nabla^2\varphi^1\|_{L^\infty(\partial{D})}\varepsilon^{(\frac{1}{3}-\tilde{\gamma})(d-1)}.
\end{align}
If $d\geq3$, we have
\begin{align}\label{est4}
\left|\mathcal{B}_{\tilde{u}}^{6}\right|
\leq&\int_{\partial{D}\cap \mathcal{C}_{\varepsilon^{\frac{1}{3}-\tilde{\gamma}}}}\frac{C}{|x'|^2}\cdot|\nabla_{x'}\varphi(0)||x'|+\int_{\partial{D}\cap \mathcal{C}_{\varepsilon^{\frac{1}{3}-\tilde{\gamma}}}}C\|\nabla^2\varphi\|_{L^\infty(\partial{D})}\nonumber\\
\leq&\,C|\nabla_{x'}\varphi(0)|\varepsilon^{(\frac{1}{3}-\tilde{\gamma})(d-2)}+C\|\nabla^2\varphi\|_{L^\infty(\partial{D})}\varepsilon^{(\frac{1}{3}-\tilde{\gamma})(d-1)}.
\end{align}
Hence, combining (\ref{est1})--(\ref{est4}) yields that for $\varepsilon>0$ sufficiently small, if $d=2$ and $\nabla_{x'}\varphi^1(0)=0$, 
\begin{align}\label{u1-u1*0}
&\left|\mathcal{B}_{\tilde{u}}\right|\leq C\left(|\nabla_{x'}\varphi^{2}(0)|+\|\nabla^2\varphi\|_{L^\infty(\partial{D})}\right)\varepsilon^{\frac{1}{3}-\tilde{\gamma}};
\end{align}
and if $d\geq3$,
\begin{align}\label{u1-u1*0d3}
&\left|\mathcal{B}_{\tilde{u}}\right|\leq C\left(|\nabla_{x'}\varphi(0)|+\|\nabla^2\varphi\|_{L^\infty(\partial{D})}\right)\varepsilon^{(\frac{1}{3}-\tilde{\gamma})(d-2)}.
\end{align}

We now estimate $\mathcal{B}_{w}$. It follows from Corollary \ref{corol3.4}  that
\begin{align}\label{w-1}
|\nabla w_1(x)|\leq
\frac{C}{\sqrt{\delta(x)}},\quad 0<|x'|\leq R,
\end{align}
and
\begin{align}\label{w-1*}
|\nabla w_1^*(x)|\leq
\frac{C}{|x'|},\quad 0<|x'|\leq R.
\end{align}
By definition,
\begin{align*}
\mathcal{B}_{w}&=\int_{\partial{D}\cap \mathcal{C}_{\varepsilon^{\frac{1}{3}-\tilde{\gamma}}}}\left\{\lambda \sum_{k,l=1}^d\partial_{x_k}(w_1^k-w_1^{*k})n_l(\varphi^l(x)-\varphi^l(0))\right.
\\
&\quad+\left.\mu\sum_{k,l=1}^d[\partial_{x_l}(w_1^k-w_1^{*k})+\partial_{x_k}
(w_1^l-w_1^{*l})]n_l(\varphi^k(x)-\varphi^k(0))\right\}.
\end{align*}
By (\ref{w-1}), (\ref{w-1*}) and the Taylor expansion of $\varphi^l(x)$,  
\begin{align}\label{Bw}
\left|\mathcal{B}_{w}\right|
&\leq \int_{\partial{D}\cap \mathcal{C}_{\varepsilon^{\frac{1}{3}-\tilde{\gamma}}}}
\frac{C}{|x'|}\cdot\left(|\nabla_{x'}\varphi^l(0)||x'|+\|\nabla^2\varphi^l\|_{L^\infty(\partial{D})}|x'|^2\right)\nonumber\\
&\leq C|\nabla_{x'}\varphi^l(0)|\varepsilon^{(\frac{1}{3}-\tilde{\gamma})(d-1)}+\|\nabla^2\varphi^l\|_{L^\infty(\partial{D})}\varepsilon^{(\frac{1}{3}-\tilde{\gamma})d}\nonumber\\
&\leq\,C\left(|\nabla_{x'}\varphi(0)|+\|\nabla^2\varphi\|_{L^\infty(\partial{D})}\right)\varepsilon^{(\frac{1}{3}-\tilde{\gamma})(d-1)}.
\end{align}
This, together with (\ref{u1-u1*0}), implies that, for $d=2$, if $\nabla_{x'}\varphi^1(0)=0$,
\begin{align*}
\left|\mathcal{B}^{in}\right|\leq|\mathcal{B}_{\tilde{u}}|+|\mathcal{B}_{w}|
\leq C\left(|\nabla_{x'}\varphi(0)|+\|\nabla^2\varphi\|_{L^\infty(\partial{D})}\right)\varepsilon^{\frac{1}{3}-\tilde{\gamma}}.
\end{align*}
Combining with \eqref{B_out}, we now simply choose $\tilde{\gamma}=\gamma_{2}=1/6$, such that $\frac{1}{3}-\gamma_{2}=\gamma_{2}$. Thus, we have, for $d=2$,
\begin{align*}
|b_1-b_1^*|\leq|\mathcal{B}^{in}|+|\mathcal{B}^{out}|\leq\,C\left(|\nabla_{x'}\varphi^{2}(0)|+\|\nabla^2\varphi\|_{L^\infty(\partial{D})}\right)\varepsilon^{1/6}.
\end{align*}
For $d\geq3$, combining \eqref{Bw} together with (\ref{u1-u1*0d3}) yields that
\begin{align*}
\left|\mathcal{B}^{in}\right|\leq|\mathcal{B}_{\tilde{u}}|+|\mathcal{B}_{w}|
\leq C\left(|\nabla_{x'}\varphi(0)|+\|\nabla^2\varphi\|_{L^\infty(\partial{D})}\right)\varepsilon^{(\frac{1}{3}-\tilde{\gamma})(d-2)}.
\end{align*}
Therefore, using \eqref{B_out} again and picking $\tilde{\gamma}=\gamma_{d}=\frac{d-2}{3(d-1)}$ (such that $(\frac{1}{3}-\gamma_{d})(d-2)=\gamma_{d}$), we have, for $d\geq3$,
\begin{align*}
|b_1-b_1^*|\leq|\mathcal{B}^{in}|+|\mathcal{B}^{out}|\leq\,C\left(|\nabla_{x'}\varphi(0)|+\|\nabla^2\varphi\|_{L^\infty(\partial{D})}\right)\varepsilon^{\gamma_{d}}.
\end{align*}
The proof of Lemma \ref{lemma5.1} is completed.
\end{proof}

\begin{proof}[Proof of Theorem \ref{thm1.2}]  Under the assumptions of Theorem \ref{thm1.2} that $b_{k_{0}}^{*}\neq0$ for some integer $1\leq\,k_0\leq\,d$, it follows from Lemma \ref{lemma5.1} that there exists a universal constant $C_{0}>0$ and a sufficiently small number $\varepsilon_0>0$, such that, for $0<\varepsilon<\varepsilon_0$,
\begin{align}\label{bk0}
|b_{k_0}|>\frac{C_{0}}{2}>0.
\end{align}

By the definition of  $A^*=(a_{\alpha\beta}^*)_{d\times{d}}$, where $a_{\alpha\beta}^*$ is the cofactor of $a_{\alpha\beta}$, and  Lemma \ref{lemma4.2}, we have
 \begin{align}\label{a*}a_{\alpha\alpha}^*\sim \frac{1}{(\rho_d(\varepsilon))^{d-1}},
 ~\alpha=1,\cdots,d;\quad a_{\alpha\beta}^*\sim \frac{1}{(\rho_d(\varepsilon))^{d-2}},\quad\alpha\neq\beta.\end{align}
According to (\ref{4.18}),  (\ref{errors}), (\ref{X1}),  (\ref{bk0}) and (\ref{a*}), for sufficiently small $\varepsilon$, 
 \begin{align}\label{C-k0}
 |C^{k_0}-\varphi^{k_0}(0)|&=\left|\frac{1}{\det A}\Big[a_{k_0k_0}^*b_{k_0}+\sum_{\beta\neq\,k_{0}}a_{k_0\beta}^*b_\beta\Big]+Errors\right|\nonumber\\
 &\geq\frac{1}{2}\frac{1}{\det A}a_{k_0k_0}^*|b_{k_0}|\geq\frac{\rho_d(\varepsilon)}{C}.
 \end{align}

On the other hand, in view of Corollary \ref{corol3.4}, we obtain
\begin{align}\label{u-k01}
|\partial_{x_{d}}u_{k_0}^{k_0}|&=|\partial_{x_{d}}\tilde{u}_{k_0}^{k_0}+\partial_{x_{d}}(u_{k_0}^{k_0}-\tilde{u}_{k_0}^{k_0})|\nonumber\\
&\geq|\partial_{x_{d}}\tilde{u}_{k_0}^{k_0}|-|\partial_{x_{d}}(u_{k_0}^{k_0}-\tilde{u}_{k_0}^{k_0})|\geq\frac{1}{C(\varepsilon+|x'|^2)},\quad x\in\Omega_R.
\end{align}
At the same time, since $\tilde{u}_\alpha^{k_0}=0$ if $\alpha\neq\,k_{0}$, it is easy to see from Corollary \ref{corol3.4} that
\begin{align}\label{u-k02}
|\partial_{x_{d}}u_\alpha^{k_0}|&=|\partial_{x_{d}}\tilde{u}_\alpha^{k_0}+\partial_{x_{d}}(u_\alpha^{k_0}-\tilde{u}_\alpha^{k_0})|\nonumber\\
&=|\partial_{x_{d}}(u_\alpha^{k_0}-\tilde{u}_\alpha^{k_0})|\leq\frac{C }{\sqrt{\varepsilon+|x'|^2}},\quad\alpha\neq k_0,\quad x\in\Omega_R.
\end{align}
Therefore, by a combination of the estimates \eqref{C-k0}, (\ref{u-k01}), (\ref{u-k02}), we get,  for $(0',x_{d})\in\overline{P_{1}P}$,
\begin{align}\label{ck}
\left|\sum_{\alpha=1}^d(C^{\alpha}-\varphi^{\alpha}(0))\nabla u_\alpha\right|&\geq\left|\sum_{\alpha=1}^d(C^{\alpha}-\varphi^{\alpha}(0))\partial_{x_d} u_\alpha^{k_0}\right|\nonumber\\
&\geq\left|(C^{k_0}-\varphi^{k_0}(0))\partial_{x_d} u_{k_0}^{k_0}\right|-\left|\sum_{\alpha\neq k_0}^d(C^{\alpha}-\varphi^{\alpha}(0))\partial_{x_d} u_\alpha^{k_0}\right|\nonumber\\
&\geq\frac{\rho_d(\varepsilon)}{C\varepsilon}.
\end{align}
Here we used the assumption that $b_{\alpha}^{*}=0$ for $\alpha\neq\,k_{0}$, $1\leq\alpha\leq\,d$ when $d\geq4$.
By means of Corollary \ref{corol2.2}, \eqref{C3}, Lemma \ref{lemma4.2} and Lemma \ref{lemma4.3}, especially for $x=(0',x_{d})\in\overline{P_{1}P}$,
\begin{align}
&\left|\sum_{\alpha=d+1}^{\frac{d(d+1)}{2}}C^{\alpha}\nabla u_{\alpha}\right|\leq\frac{C(\varepsilon+|x'|)}{\varepsilon+|x'|^2}\leq\,C,\label{thm1.2-2}\\
&|\nabla u_0|\leq \frac{C|\nabla\varphi(0)||x'|}{\varepsilon+|x'|^2}+C\leq\,C.\label{thm1.2-3}
\end{align}
 Combining (\ref{ck}), (\ref{thm1.2-2}), (\ref{thm1.2-3}) and \eqref{decomposition_nablau} immediately yields that for $x=(0',x_{d})\in\overline{P_{1}P}$,
\begin{align*}
|\nabla u(0', x_{d})|\geq\frac{\rho_d(\varepsilon)}{C\varepsilon},\quad 0<x_{d}<\varepsilon.
\end{align*}
Theorem \ref{thm1.2} is thus established.
\end{proof}

\vspace{1cm}

\noindent{\bf{\large Acknowledgements.}} H.G. Li would like to thank Professor YanYan Li for his encouragements and constant supports.

\bibliographystyle{plain}

\def\cprime{$'$}

\end{document}